\documentclass[12pt]{amsart}
\usepackage[latin9]{inputenc}
\usepackage{xcolor}
\usepackage{amstext}
\usepackage{amsthm}
\usepackage{amssymb}
\usepackage[unicode=true,
 bookmarks=false,
 breaklinks=false,pdfborder={0 0 1},backref=section,colorlinks=false]{hyperref}
\usepackage{xcolor}
\usepackage{verbatim}
\usepackage{amstext}
\usepackage{amsthm}
\usepackage{eurosym}
\usepackage{latexsym}
\usepackage{amsthm}
\usepackage{amsfonts}
\usepackage{ulem}

\makeatletter
\numberwithin{equation}{section}
\numberwithin{figure}{section}
\numberwithin{equation}{section}
\numberwithin{figure}{section}
\numberwithin{equation}{section}
\newtheorem{theorem}{Theorem}

\newtheorem{lemma}[theorem]{Lemma}
\newtheorem{corollary}[theorem]{Corollary}
\newtheorem{proposition}[theorem]{Proposition}

\theoremstyle{definition}
\newtheorem{definition}[theorem]{Definition}

\newtheorem*{acknowledgements*}{Acknowledgements}
\theoremstyle{remark}
\newtheorem{remark}[theorem]{Remark}

\numberwithin{theorem}{section}   
\thanks{The work of the first author was completed as a part of the implementation of the development program of the Scientific and Educational Mathematical Center Volga Federal District, agreement no. 075-02-2021-1393.}
\thanks{}
\thanks{}
\subjclass[2010]{Primary 46B70; Secondary 46E30, 46M35, 46A45}
\keywords{interpolation space, $\ell^p$-spaces, quasi-Banach space, quasi-norm group, ${\mathcal K}$-functional, 
Calder\'{o}n-Mityagin property, orbit}
\makeatother

\begin{document}
\title[Arazy-Cwikel and Calderón-Mityagin type properties ]{Arazy-Cwikel and
Calderón-Mityagin type properties of the couples $(\ell^{p},\ell^{q})$, $
0\le p<q\le\infty$}
\author{Sergey V. Astashkin}
\address{Astashkin:Department of Mathematics, Samara National Research
University, Moskovskoye shosse 34, 443086, Samara, Russia}
\email{astash56@mail.ru}
\author{Michael Cwikel}
\address{Cwikel: Department of Mathematics, Technion - Israel Institute of
Technology, Haifa 32000, Israel }
\email{mcwikel@math.technion.ac.il}
\author{Per G. Nilsson}
\address{Nilsson: Roslagsgatan 6, 113 55 Stockholm, Sweden}
\email{pgn@plntx.com}
\date{\today }

\begin{abstract}
We establish Arazy-Cwikel type properties for the family of couples $
(\ell^{p},\ell^{q})$, $0\le p<q\le\infty$, and show that $(\ell^{p},\ell^{q})
$ is a Calder\'{o}n-Mityagin couple if and only if $q\ge1$. Moreover, we
identify interpolation orbits of elements with respect to this couple for
all $p$ and $q$ such that $0\le p<q\le\infty$ and obtain a simple positive
solution of a Levitina-Sukochev-Zanin problem, clarifying its connections
with whether $(\ell^{p},\ell^{q})$ has the Calder\'{o}n-Mityagin property or
not.
\end{abstract}

\maketitle

\section{Introduction}

\label{Intro}

Nowadays, the interpolation theory of operators is rather completely
presented in several excellent books; see, for example, Bergh and L{ö}fstr{ö}
m \cite{BL76}, Bennett and Sharpley \cite{BSh}, Brudnyi and Kruglyak \cite
{BK81}, Krein, Petunin and Semenov \cite{KPS82}, Triebel \cite{Triebel}. In
these monographs the reader can find not only a systematic treatment of
problems within the theory itself, but also valuable applications of
interpolation methods and results to various other fields of mathematics.
Let us also mention several books which contain applications of
interpolation theory to a variety of fields: in \cite{Graf09,HNVW17} there are applications to harmonic analysis, in \cite
{LT79-I,LT79-II,wojtaszczyk} to Banach space theory, and in \cite{KS,Ast20}
to classical systems in $L^{p}$ spaces and in other rearrangement invariant
spaces. Furthermore, the papers \cite{LSZ17,LSZ17a} (and the references
therein) include applications to noncommutative analysis, and, finally, the
survey \cite{KSM03} contains a very attractive account of the interaction
between interpolation theory and the geometry of Banach spaces. Of course,
the above list is far from being complete.

\vskip0.2cm

One of the reasons for there being such fruitful applicability of
interpolation theory is that, for many couples $(X_{0},X_{1})$, we can
effectively describe the class $Int(X_{0},X_{1})$ of all interpolation
spaces. In most of the known cases of couples $\left(X_{0},X_{1}\right)$ for
which this is possible, this description is formulated by using the Peetre $\mathcal{K}$-functional, which plays an important role in the theory. For
those couples the terminology \textit{Calder\'{o}n couple} or \textit{Calder\'{o}n-Mityagin couple} is often used. This is because the first example of such
a couple was obtained by Calder\'{o}n \cite{CAL} and Mityagin \cite{Mit}. They
proved independently that a Banach function space $X$ on an arbitrary
underlying measure space is an interpolation space with respect to the
couple $\left(L^{1},L^{\infty}\right)$ on that measure space if and only if
the following monotonicity property\footnote{
Here we are using Calder\'{o}n's terminology. Mityagin formulates this result
somewhat differently.} holds: if $f\in X$, $g\in L^{1}+L^{\infty}$ and 
\begin{equation*}
\int\nolimits _{0}^{t}g^{\ast}\left(s\right)\,ds\leq\int\nolimits
_{0}^{t}f^{\ast}\left(s\right)\,ds,\;\;t>0 
\end{equation*}
(where $h^{*}$ denotes the nonincreasing left-continuous rearrangement of $
|h|$), then $g\in X$ and $\left\Vert g\right\Vert _{X}\leq\left\Vert
f\right\Vert _{X}.$ Peetre \cite {PeetreJ1963Nouv,PeetreJ1968brasilia} had proved (cf.~also a similar result due
independently to Oklander \cite{OklanderE1964}, and cf. also \cite[
pp.~158--159]{KreeP1968}) that the functional $\int\nolimits
_{0}^{t}f^{\ast}\left(s\right)\,ds$ is in fact the $\mathcal{K}$-functional
of the function $f\in L^{1}+L^{\infty}$ for the couple $\left(L^{1},L^{
\infty}\right)$. So the results of \cite{CAL,Mit} naturally suggested the
form that analogous results for couples other than $\left(L^{1},L^{\infty}
\right)$ might take, expressed in terms of the $\mathcal{K}$-functional for
those couples. This led many mathematicians to search for such analogous
results. Let us mention at least some of the many results of this kind which
were obtained: Lorentz-Shimogaki\footnote{
This paper also describes the interpolation spaces for the couple $
(L^{1},L^{p})$ but not in explicit terms of the $\mathcal{K}$-functional for
that couple.} \cite{LSh} ($\left(L^{p},L^{\infty}\right),1<p<\infty$),
Sedaev-Semenov \cite{SeSe71} (couples of weighted $L^{1}$-spaces), Dmitriev 
\cite{Dmi75} (relative interpolation of couples $
\left(L^{1}(w_{0}),L^{1}(w_{1})\right)$ and $\left(L^{1},L^{\infty}\right)$
), Peetre \cite{PeetreJ1971x} (relative interpolation of an \textit{arbitrary} Banach couple with a couple of weighted $L^{\infty}$-spaces) (see also 
\cite[p.~589, Theorem 4.4.16]{BK91} or \cite[p.~29, Theorem 4.1]
{CwikelMPeetreJ1981} for this result), Sparr \cite{SP78} (couples of
weighted $L^{p}$-spaces, $1\le p\leq\infty$), and Kalton \cite{Kal92}
(couples of rearrangement invariant spaces).

\vskip0.2cm

It is worth to note, that for an arbitrary Banach couple, the uniform $
\mathcal{K}$-monotone interpolation spaces, which are closely related to the
Calder\'{o}n-Mityagin property (see Definition \ref{K-monotonedef} and Remark 
\ref{rem:CMcoupleAndKmontoneSpaces} below) can also be described in a more
concrete way. This important fact is due to Brudnyi and Kruglyak \cite
{BK81,BK91} and follows from their proof (see \cite[pp.~503-504]{BK91}) of a
conjecture due to S. G. Krein \cite{DmKrOv77}. One of its consequences is
that, for every Banach couple $(X_{0},X_{1})$ with the Calder\'{o}n-Mityagin
property, the family $Int(X_{0},X_{1})$ of all its interpolation spaces
can be parameterized by the set of so-called $\mathcal{K}$\textit{-method
parameters}.

\vskip0.2cm

Moreover, using the Calder\'{o}n-Mityagin property\footnote{
In fact the arguments used in \cite{Arazy-Cwikel} also yield a shorter and
simpler proof of this property, at least for couples of exponents $p$ in the
range $1\le p\le\infty$.} of couples of $L^{p}$-spaces in the range $1\le
p\le\infty$, Arazy and Cwikel proved, in \cite{Arazy-Cwikel}, that for all $
1\leq p<q\leq\infty$ and for each underlying measure space 
\begin{equation}
Int\left(L^{p},L^{q}\right)=Int\left(L^{1},L^{q}\right)\cap
Int\left(L^{p},L^{\infty}\right).  \label{AC-statement-1}
\end{equation}
Later on, Bykov and Ovchinnikov obtained a similar result for families of
interpolation spaces, corresponding to weighted couples of shift-invariant
ideal sequence spaces \cite{BO06}.

\vskip0.2cm

On the negative side, Ovchinnikov and Dmitriev \cite{OvcDm75} showed that
the couple $(\ell ^{1}(L^{1}),\ell ^{1}(L^{\infty }))$ of vector-valued
sequences is not a Calder\'{o}n-Mityagin couple. Neither is the couple $\left(
L^{p},W^{1,p}\right) $ when $p\in (2,\infty )$. (See \cite[p.~218]
{CwikelM1976Temp}.) Later on, Ovchinnikov \cite{Ovc82} (see also \cite
{MalOvc}) proved the same result for the couple $(L^{1}+L^{\infty
},L^{1}\cap L^{\infty })$ on $(0,\infty )$. One can find more examples of
couples of rearrangement invariant spaces of this kind in Kalton's work \cite
{Kal92}. Many of these results contain a description of interpolation
orbits, which cannot be obtained by the real $\mathcal{K}$-method (see e.g. 
\cite{Ovc76}, \cite{Ovc84}, \cite{Ovc05}, \cite{Dmi81}, \cite{DMI92}). As shown
by Theorem 2.2 of \cite[pp.~36--37]{CwikelMNilssonP2003}, if $X_{0}$\ and $
X_{1}$\ are both $\sigma $-order continuous Banach lattices of measurable
functions with the Fatou property on the same underlying $\sigma $-finite
measure space $\Omega $\ and if at least one of these spaces does not
coincide to within equivalence of norm with some weighted $L^{p}$\ space on $
\Omega $. then there exist weight functions $w_{0}$\ and $w_{1}$\ on $\Omega 
$\ for which the couple of weighted lattices $\left(
X_{0,w_{0}},X_{1,w_{1}}\right) $\ is not a Calderón-Mityagin couple.

We refer to the article \cite{CwikelM-Encyclopedia-1} for additional details
about Calder\'{o}n-Mityagin couples.

\vskip0.2cm

All of the results listed above were obtained for couples of \textit{Banach}
{} spaces. But there were also some ventures beyond Banach couples. In \cite
{SP78} Sparr was in fact also able to treat couples of weighted $L^{p}$
spaces for $p\in(0,\infty)$ under suitable hypotheses, and then Cwikel \cite
{Cwikel1} considered the couple $\left(\ell^{p},\ell^{\infty}\right)$ also
for $p$ in this extended range. New questions have recently arisen (see, for
instance, \cite{LSZ17a}, \cite{CN17}) that require analogous results for
more general situations, say, for quasi-Banach couples or even for couples
of quasi-normed Abelian groups. The extension of the basic concepts and
constructions of interpolation theory to the latter setting was initiated
long ago by Peetre and Sparr in \cite{PS}.

\vskip0.2cm

We recall that $\ell ^{0}$ is the linear space (sometimes considered merely
as an Abelian group) of all eventually zero sequences $x=(x_{k})_{k=1}^{
\infty }$, equipped with the \textquotedblleft norm\textquotedblright 
\footnote{Although $\Vert \cdot \Vert _{\ell ^{0}}$ does not satisfy the condition $\Vert \lambda f\Vert _{\ell ^{0}}$=$\left\vert \lambda \right\vert \Vert f\Vert _{\ell^{0}}$ for scalars $\lambda $, it is a $\left( 1,\infty \right) $-norm or $\infty $-norm on the Abelian group $\ell ^{0}$ in the terminology of \cite[p.~219]{PS}.}\textbf{\ }
$\Vert x\Vert _{\ell ^{0}}:=\mathrm{card}(\mathrm{
supp}\,x)$, where $\mathrm{supp}\,x$ is the support of $x$. This space is an
analogue of the space or normed Abelian group $L^{0}$, {\ which consists of
all measurable functions on $(0,\infty )$ with supports of finite measure,
equipped with the quasi-norm $\Vert f\Vert _{L^{0}}:=m\{t>0:\,f(t)\neq 0\}$ (
$m$ is the Lebesgue measure) and of the space of operators $\mathfrak{S}_{0}\left(A,B\right) $ introduced on p. 249 and p. 256 respectively of \cite{PS}.}
Comparing some simple calculations with $L^{0}$ in \cite{PS} with some
quantities appearing implicitly in \cite{LSZ17,LSZ17a} and \cite{CN17} can
lead one to understand that $\ell ^{0}$ can play a useful role in studying
interpolation properties of $\ell ^{p}$ spaces for $p>0$. Note also that
independently $\ell ^{0}$ appeared explicitly in \cite{Ast-20}, where a
description of orbits of elements in the couple $(\ell ^{0},\ell ^{1})$ is given.

\vskip0.2cm

The two main aims of this paper are, first of all, to completely determine
for which values of $p$ and $q$ in the range $0\leq p<q\leq \infty $, the
couple $\left( \ell ^{p},\ell ^{q}\right) $ has the Calderón-Mityagin
property and then, secondly, to extend a property analogous to the
Arazy-Cwikel property \eqref{AC-statement-1} to the couples $(\ell ^{p},\ell
^{q})$, with $p$ and $q$ in the enlarged range $0\leq p<q\leq \infty $ and
with the role of $L^{1}$ in \eqref{AC-statement-1} now played by $\ell ^{0}$.

\vskip0.2cm

Note that there are close connections between the present paper and the paper \cite{CSZ}. Although \cite{CSZ} mainly considers the couples $\left(L^{p},L^{q}\right) $ of function spaces on $(0,\infty )$, it also deals with interpolation properties of the analogous sequence space couples $\left(\ell ^{p},\ell ^{q}\right) $ for the range $0\leq p<q\leq \infty $.
However, in contrast to our paper, the authors of \cite{CSZ} restrict
themselves to studying the Calderón-Mityagin case,\textbf{\ } i.e., for
values $q\geq 1$. It seems that some of the results in \cite{CSZ} for this
case could be used to establish some of our results, and vice versa. We
shall comment more explicitly about connections with \cite{CSZ} at
appropriate places in our text, however we have kept our approach almost
self-contained.

\vskip0.2cm

The couple $(\ell ^{0},\ell ^{\infty })$ has some advantages over the
corresponding Banach couple $(\ell ^{1},\ell ^{\infty })$. In particular, as
remarked in \cite{CSZ}, it is well-known that there exist symmetric Banach
sequence spaces, which are \textit{not} interpolation spaces with respect to
the latter couple (see e.g. \cite[Example 2.a.11, p.~128]{LT79-II}. In
contrast to that, every symmetric quasi-Banach sequence space $E$ {\ is an interpolation space with respect to the couple $\left( \ell ^{0},\ell^{\infty }\right)$ 
(this can be obtained by obvious modifications of reasoning in the papers \cite{HM-90} and \cite{Ast-94}, where the analogous property is proved for the couple $\left( L^{0},L^{\infty }\right)$ on $(0,\infty)$ and rearrangement invariant quasi-Banach function spaces).} 

\vskip0.2cm

Some other partial results for the couples $(\ell ^{p},\ell ^{q})$, in the
non-Banach case, were obtained more recently in \cite
{CN17,LSZ17a,CSZ,Ast-20,Cad}. Moreover, in \cite{CSZ}, the above
Arazy-Cwikel property has been proved for the couple $(L^{0},L^{\infty })$
of measurable functions on the semi-axis $(0,\infty )$ with the Lebesgue
measure. 
Observe however that there are differences in the properties of the
quasi-Banach spaces $\ell ^{p}$ and $L^{p}$ that are essential in our
context; for instance, if $p\in (0,1)$, then $(\ell ^{p})^{\ast }=\ell ^{1}$
while $(L^{p})^{\ast }=\{0\}$ (see Section \ref{Prel-sequence-spaces}).

\vskip0.2cm

In general, the above-mentioned Brudnyi-Kruglyak result about a description of all interpolation spaces with respect to Calder\'{o}n-Mityagin couples of Banach spaces cannot be extended
to the class of quasi-Banach couples. Nevertheless, whenever $p$ and $q$ are
such that the couple $\left(\ell^{p},\ell^{q}\right)$ is a 
Calder\'{o}n-Mityagin couple (including in the non-Banach case), then every
interpolation space with respect to $\left(\ell^{p},\ell^{q}\right)$ can be
described by using the real $\mathcal{K}$-method of interpolation. Moreover, but discussion of this is deferred to a forthcoming paper \cite{AN-21}, a similar result holds for a rather wide subclass of quasi-Banach couples (the latter paper will also deal with some other related problems).

\vskip0.2cm

Let us describe now the main results of the paper in more detail. In Section
2, we give preliminaries with basic definitions and results. So, we address some versions of the Holmstedt inequality and give descriptions of
the ${\mathcal{K}}$- and ${\mathcal{E}}$-functionals for couples of $\ell^{p}
$-spaces. Section 3 contains some auxiliary (apparently well-known) results,
in particular, an extrapolation theorem for operators bounded on $\ell^{p}$, 
$0<p<1.$

\vskip0.2cm

The central result of the next section is Theorem \ref{int spaces}, which
extends the above mentioned Arazy-Cwikel theorem \eqref{AC-statement-1} to
the sequence space setting, showing that 
\begin{equation}
Int\left( \ell ^{p},\ell ^{q}\right) =Int\left( \ell ^{s},\ell ^{q}\right)
\cap Int\left( \ell ^{p},\ell ^{r}\right)   \label{AC-statement-2}
\end{equation}
for all $0\leq s<p<q<r\leq \infty $. It is worth noting that 
\eqref{AC-statement-2} holds, in particular, in the range $0\leq s<p<q<r<1$,
i.e., when all the couples involved in \eqref{AC-statement-2} fail to have
the Calder\'{o}n-Mityagin property. This fact indicates that Arazy-Cwikel type
properties of couples do not imply that they are Calder\'{o}n-Mityagin ones.
Observe that a closely related result, {\ under the additional restriction $
q\geq 1$,} has been proved in \cite[Theorem~5.6]{CSZ}.

\vskip0.2cm

The main ingredient in the proof of relations \eqref{AC-statement-2} is Theorem \ref{p3}, which has also other interesting consequences. The first of them,
Corollary \ref{Ar-Cw}, states that the condition $q\geq1$ ensures that $
\left(\ell^{p},\ell^{q}\right)$ is a uniform Calder\'{o}n-Mityagin couple. The
second, Corollary \ref{Corr-1}, presents a complete description of
interpolation orbits of elements of the space $\ell^{q}$ with respect to the
couple $\left(\ell^{p},\ell^{q}\right)$ both in the cases $q\geq1$ and $q<1$
. Let us mention also the result of Theorem \ref{int properties}, which is a
consequence of Theorem \ref{int spaces} combined with a self-improvement property of interpolation between the spaces $\ell^{0}$ and $\ell^{q}$.
Theorem \ref{int properties} states that, if $E\in
Int\left(\ell^{0},\ell^{q}\right)$, then there exists $p>0$ such that $E\in
Int\left(\ell^{p},\ell^{q}\right)$. Hence, interpolation of quasi-Banach
spaces with respect to the couple $\left(\ell^{0},\ell^{q}\right)$ can be
reduced, in fact, to that with respect to the couples $\left(\ell^{p},
\ell^{q}\right)$ with $p>0$. This phenomenon allows us to obtain rather simply, in the case $q\geq1$, the positive answer to the Levitina-Sukochev-Zanin conjecture, which was posed in \cite{LSZ17a} and resolved in \cite{CSZ} (its earlier version in majorization terms may be found in the preprint \cite{LSZ17}).
Moreover, we reveal its connections with the Calder\'{o}n-Mityagin property of
the couple $(\ell^{p},\ell^{q})$, showing that the answer to the latter
conjecture is negative if $0<q<1$.

\vskip0.2cm

In Section 5 we prove that $\left(\ell^{p},\ell^{q}\right)$ is not a Calder\'{o}n-Mityagin couple whenever $0\leq p<q<1$ (see Theorem \ref
{Th:No-CM-Property}). In fact, we obtain a stronger result, which reads that for every $
g\in\ell^{q}\setminus\ell^{p}$ there exists $f\in\ell^{q}$ satisfying the
condition 
\begin{equation*}
{\mathcal{K}}\left(t,g;\ell^{p},\ell^{q}\right)\leq{\mathcal{K}}
\left(t,f;\ell^{p},\ell^{q}\right),\;\;t>0, 
\end{equation*}
but $g\neq Tf$ for every linear operator $T$ bounded in $\ell^{p}$ and $\ell^{q}$. Combining Theorem \ref{Th:No-CM-Property} with Corollary \ref{Ar-Cw}, we conclude that $\left(\ell^{p},\ell^{q}\right)$ is a uniform
Calder\'{o}n-Mityagin couple if and only if $q\geq1$.

\vskip0.2cm

Considering the above-mentioned Levitina-Sukochev-Zanin conjecture, Cwikel
and Nilsson have introduced, in \cite{CN17}, the so-called $S_{q}$-property
expressed in terms of a majorization inequality. In Section 6 we show that
for every $q\ge1$ a quasi-Banach sequence space $E$ has the $S_{q}$-property
if and only if $E\in Int(\ell^{0},\ell^{q})$ (see Theorem \ref{S_q-prop} and
Corollary \ref{S_q-prop1}).

\vskip0.2cm

In the concluding Section 7 we prove that the couple $\left(\ell^{p},
\ell^{q}\right)$, with $0\leq p<q<1$, does not have the uniform 
Calder\'{o}n-Mityagin property (see Theorem \ref{Th:No-Bounded-CM-Property}). Clearly,
the latter result is weaker than Theorem \ref{Th:No-CM-Property} of Section
5. However, for the reader's convenience, we provide its independent proof,
which is much shorter and simpler than that of Theorem \ref
{Th:No-CM-Property}.

\vskip0.2cm

We dedicate this paper to the memory of Professor Jaak Peetre (1935 -- 2019). His profound ideas and research have played an essential role in many
mathematical topics, including those in this paper. The second and third
authors have been fortunate to have him as a very respected and close
friend, mentor and colleague for many years.

\section{Preliminaries}

\subsection{Interpolation of operators and the Calder\'{o}n-Mityagin property.}

\label{Prel-Interpolation}

Let us recall some basic constructions and definitions related to the
interpolation theory of operators. For more detailed information we refer to 
\cite{BL76,BK91,BSh,KPS82,Ovc84}.

\smallskip{}  \vskip0.2cm

In this paper we are mainly concerned with interpolation within the class of
quasi-Banach sequence spaces while linear bounded operators are considered
as the corresponding morphisms. All linear spaces considered will be over
the reals. But it should be possible to readily extend much of the theory
that we develop also to the case of complex linear spaces.

\vskip0.2cm

A pair $\vec{X}=(X_{0},X_{1})$ of quasi-Banach spaces is called a \textit{
quasi-Banach couple} if $X_{0}$ and $X_{1}$ are both linearly and
continuously embedded in some Hausdorff linear topological space. In
particular, every pair of arbitrary quasi-Banach sequence lattices $E_{0}$
and $E_{1}$ forms a quasi-Banach couple, because convergence in a
quasi-Banach sequence lattice implies coordinate-wise convergence.

\vskip0.2cm

For each quasi-Banach couple $(X_{0},X_{1})$ we define the \textit{
intersection} $X_{0}\cap X_{1}$ and the \textit{sum} $X_{0}+X_{1}$ as the
quasi-Banach spaces equipped with the quasi-norms 
\begin{equation*}
\Vert x\Vert _{X_{0}\cap X_{1}}:=\max \left\{ \Vert x\Vert
_{X_{0}}\,,\,\Vert x\Vert _{X_{1}}\right\} 
\end{equation*}
and 
\begin{equation*}
\Vert x\Vert _{X_{0}+X_{1}}:=\inf \left\{ \Vert x_{0}\Vert
_{X_{0}}\,+\,\Vert x_{1}\Vert _{X_{1}}:\,x=x_{0}+x_{1},x_{i}\in
X_{i},i=0,1\right\} ,
\end{equation*}
respectively. A linear space $X$ is called \textit{intermediate} with
respect to a quasi-Banach couple $\vec{X}=(X_{0},X_{1})$ (or is said to be 
\textit{between} $X_{0}$ and $X_{1}$) if it is a quasi-Banach space and
satisfies $X_{0}\cap X_{1}\subset X\subset X_{0}+X_{1}$ where both of these
inclusions are continuous.

\vskip0.2cm

If $\vec{X}=(X_{0},X_{1})$ and $\vec{Y}=(Y_{0},Y_{1})$ are quasi-Banach
couples, then we let $\mathfrak{L}({\vec{X}},{\vec{Y}})$ denote the space of
all linear operators $T:\,X_{0}+X_{1}\rightarrow Y_{0}+Y_{1}$ that are
bounded from $X_{i}$ in $Y_{i}$, $i=0,1$, equipped with the quasi-norm 
\begin{equation}
{\Vert T\Vert }_{\mathfrak{L}({\vec{X}},{\vec{Y}})}:=\max\limits_{i=0,1}{
\Vert T\Vert }_{X_{i}\rightarrow Y_{i}}.  \label{eq:QuasiNormLXY}
\end{equation}
In the case when $X_{i}=Y_{i}$, $i=0,1$, we simply write $\mathfrak{L}({\vec{
X}})$ or $\mathfrak{L}(X_{0},X_{1})$.

\vskip0.2cm

Let $\vec{X}=(X_{0},X_{1})$ be a quasi-Banach couple and let $X$ be an
intermediate space between $X_{0}$ and $X_{1}$. Then, $X$ is called an 
\textit{interpolation space} with respect to the couple $\vec{X}$ (or
between $X_{0}$ and $X_{1}$) if every operator $T{\in }{\mathfrak{L}}({\vec{X
}})$ is bounded on $X$. In this case, we write: $X\in Int(X_{0},X_{1})$.

\vskip0.2cm

Recall that, by the Aoki-Rolewicz theorem (see e.g. \cite[Lemma~3.10.1]{BL76}
), every quasi-Banach space is a $F$-space (i.e., the topology in that space
is generated by a complete invariant metric). In particular, this applies to
the space ${\mathfrak{L}}({\vec{X}})$ which is obviously a quasi-Banach
space with respect to the quasi-norm $T\mapsto \max \left\{ {\Vert T\Vert }
_{X_{0}\rightarrow X_{0}},{\Vert T\Vert }_{X_{1}\rightarrow X_{1}}\right\} $ (cf.~\eqref{eq:QuasiNormLXY}), and also with respect to the larger quasi-norm $T\mapsto \max \left\{ {\Vert T\Vert }_{X_{0}\rightarrow X_{0}},{
\Vert T\Vert }_{X_{1}\rightarrow X_{1}},{\Vert T\Vert }_{X\rightarrow
X}\right\} $ whenever the quasi-Banach space $X$ is an interpolation space with respect to the quasi-Banach couple $\vec{X}=(X_0,X_1)$. As is well known (see e.g. \cite[Theorem~2.2.15]
{Rud}), the Closed Graph Theorem and the equivalent Bounded Inverse Theorem
(see e.g. \cite[Corollary 2.2.12]{Rud}) hold for $F$-spaces. Therefore, by using exactly the same reasoning as required for the Banach case (see Theorem 2.4.2 of \cite[p.~28]{BL76}), we have the following: if $X$ is an interpolation quasi-Banach space with respect to a quasi-Banach couple ${\vec{X}}=(X_{0},X_{1})$, then there exists a constant $C>0$ such that ${\Vert T\Vert }_{X\rightarrow X}\leq C{\Vert T\Vert }_{{\mathfrak{L}}({\vec{X}})}$ for every $T\in {\mathfrak{L}}({\vec{X}})$. The least constant $C$, satisfying the last inequality for all such $T$, is called the \textit{interpolation constant} of $X$ with
respect to the couple $\vec{X}$.

\vskip0.2cm

One of the most important ways of constructing interpolation spaces is based
on use of the \textit{Peetre ${\mathcal{K}}$-functional}, which is defined
for an arbitrary quasi-Banach couple $(X_{0},X_{1})$, for every $x\in
X_{0}+X_{1}$ and each $t>0$ as follows: 
\begin{equation*}
{\mathcal{K}}(t,x;X_{0},X_{1}):=\inf
\{||x_{0}||_{X_{0}}+t||x_{1}||_{X_{1}}:\,x=x_{0}+x_{1},x_{i}\in {X_{i}}\}.
\end{equation*}
For each fixed $x\in X_{0}+X_{1}$ one can easily show that the function $
t\mapsto {\mathcal{K}}(t,x;X_{0},X_{1})$ is continuous, non-decreasing,
concave and non-negative on $\left( 0,\infty \right) $ \cite[Lemma~3.1.1]
{BL76}. On the other hand, for each fixed $t>0$, the functional $x\mapsto {
\mathcal{K}}(t,x;X_{0},X_{1})$ is an equivalent quasi-norm on $X_{0}+X_{1}$.

\vskip0.2cm

As already discussed at some length in the introduction, for quite a large
class of (quasi-)Banach couples, the ${\mathcal{K}}$-functional can be used
to describe \textit{all} interpolation (quasi-)Banach spaces with respect to
those couples. We first need the following definition:

\begin{definition}
\label{K-monotonedef} Let $X$ be an intermediate space with respect to a
quasi-Banach couple $\vec{X}=\left(X_{0},X_{1}\right)$. Then, $X$ is said to
be a \textit{${\mathcal{K}}$-monotone space} with respect to this couple if
whenever elements $x\in X$ and $y\in X_{0}+X_{1}$ satisfy 
\begin{equation*}
{\mathcal{K}}\left(t,y;X_{0},X_{1}\right)\leq{\mathcal{K}}
\left(t,x;X_{0},X_{1}\right),\;\text{for all\thinspace}\;t>0, 
\end{equation*}
it follows that $y\in X$. If it also follows that $\left\Vert y\right\Vert
_{X}\leq C\left\Vert x\right\Vert _{X}$, for a constant $C$ which does not
depend on $x$ and $y$, then we say that $X$ is a \textit{uniform ${\mathcal{K
}}$-monotone} space with respect to the couple $\vec{X}$. The infimum of all
constants $C$ with this property is referred as the \textit{${\mathcal{K}}$
-monotonicity constant} of $X$. Clearly, each ${\mathcal{K}}$-monotone space
with respect to the couple $\vec{X}$ is an interpolation space between $X_{0}
$ and $X_{1}$.
\end{definition}

Note that every ${\mathcal{K}}$-monotone Banach space with respect to a couple of Banach lattices is also a uniform ${\mathcal{K}}$-monotone  space with respect to this couple \cite[Theorem 6.1]{CwikelMNilssonP2003}. 

\begin{definition}
\label{def:Orbit} Let $\vec{X}=\left(X_{0},X_{1}\right)$ and $\vec{Y}
=\left(Y_{0},Y_{1}\right)$ be two quasi-Banach couples and let $x\in
X_{0}+X_{1}$, $x\neq0$. The \textit{orbit} ${\mathrm{Orb}}
_{\left(X_{0},X_{1}\right)}(x;Y_{0},Y_{1})$ of $x$ with respect to the class
of operators $\mathfrak{L}(\vec{X},\vec{Y})$ is the linear space 
\begin{equation*}
\left\{ Tx:\,T\in\mathfrak{L}(\vec{X},\vec{Y})\right\} . 
\end{equation*}
This space may be equipped with the quasi-norm defined by 
\begin{equation*}
\left\Vert y\right\Vert _{\mathrm{Orb}\left(x\right)}:=\inf\left\{
\left\Vert T\right\Vert _{\mathfrak{L}(\vec{X},\vec{Y})}:\;y=Tx,T\in
\mathfrak{L}(\vec{X},\vec{Y})\right\} . 
\end{equation*}

In the case when $\left( X_{0},X_{1}\right) =\left( Y_{0},Y_{1}\right) $ we
will use the shortened notation ${\mathrm{Orb}}\left( x;X_{0},X_{1}\right) $.
\end{definition}

\vskip0.2cm

\bigskip {} Since any orbit $Orb\left( x;X_{0},X_{1}\right) $ can be
regarded as a quotient of the quasi-Banach space $\mathfrak{L}( 
\overrightarrow{X})$, it is a quasi-Banach space. If for every nonzero 
$x\in X_{0}+X_{1}$ there exists a linear functional $x^{\ast }\in \left(
X_{0}+X_{1}\right) ^{\ast }$ with $\left\langle x,x^{\ast }\right\rangle
\neq 0$ then $X_{0}\cap X_{1}$ is contained in $Orb\left(
x;X_{0},X_{1}\right) $ continuously (see e.g. \cite[Section 1.6, p. 368]
{Ovc84}). It is easy to see that then, moreover, each orbit $Orb\left(
x;X_{0},X_{1}\right) $ is an interpolation space between $X_{0}$ and $X_{1}$.

\vskip0.2cm

A similar concept may be defined by using the ${\mathcal{K}}$-functional.

\begin{definition}
\label{def:K-orbit} Let $\vec{X}=\left(X_{0},X_{1}\right)$ and $\vec{Y}
=\left(Y_{0},Y_{1}\right)$ be two quasi-Banach couples. The $\mathit{{
\mathcal{K}}-orbit}$ of an element $x\in X_{0}+X_{1}$, $x\neq0$, which we
denote by ${{\mathcal{K}}-{\mathrm{Orb}}}_{\left(X_{0},X_{1}\right)}
\left(x;Y_{0},Y_{1}\right)$ is the space of all $y\in Y_{0}+Y_{1}$ such that
the following quasi-norm 
\begin{equation*}
\left\Vert y\right\Vert _{{{\mathcal{K}}-{\mathrm{Orb}}}(x)}:=\sup_{t>0}
\frac{{\mathcal{K}}\left(t,y;Y_{0},Y_{1}\right)}{{\mathcal{K}}
\left(t,x;X_{0},X_{1}\right)} 
\end{equation*}
is finite. If $\left(X_{0},X_{1}\right)=\left(Y_{0},Y_{1}\right)$, then we
simplify the above notation to ${{\mathcal{K}}-{\mathrm{Orb}}}
\left(x;X_{0},X_{1}\right)$.
\end{definition}

{One can easily check that each $\mathcal{K}$-orbit of an element $x\in
X_{0}+X_{1}$, $x\neq 0$, is an interpolation quasi-normed space between $
X_{0}$ and $X_{1}$.}

\vskip0.2cm

It is obvious that for every quasi-Banach couple $\left(X_{0},X_{1}\right)$ and each $x\in X_{0}+X_{1}$ we have 
\begin{equation*}
{\mathrm{Orb}}(x;X_{0},X_{1}){\subset}{\mathcal{K}}-{\mathrm{Orb}}(x;X_{0},X_{1})
\end{equation*}
with constant $1$.

\begin{definition}
\label{def:CalderonMityaginCouple} A quasi-Banach couple $\vec{X}=\left(
X_{0},X_{1}\right) $ is said to be a \textit{Calder\'{o}n-Mityagin couple} (or
to have the \textit{Calder\'{o}n-Mityagin property}) if for each $x\in
X_{0}+X_{1}$ 
\begin{equation*}
{{\mathcal{K}}-{\mathrm{Orb}}}(x;X_{0},X_{1}){\subset }{\mathrm{Orb}}
(x;X_{0},X_{1}),  
\end{equation*}
i.e., if for every $y\in {{\mathcal{K}}-{\mathrm{Orb}}}(x;X_{0},X_{1})$
there exists an operator $T\in \mathfrak{L}(\vec{X})$ such that $y=Tx$. If
additionally we can choose $T\in \mathfrak{L}(\vec{X})$ so that $\Vert
T\Vert _{\mathfrak{L}(\vec{X})}\leq C\left\Vert y\right\Vert _{{{\mathcal{K}}
-{\mathrm{Orb}}}(x)}$, where $C$ is independent of $x$ and $y$, then $\vec{X}
$ is called a \textit{uniform Calder\'{o}n-Mityagin couple} (or we say that $
\vec{X}$ has the \textit{uniform Calder\'{o}n-Mityagin property}).
\end{definition}

The name of the last property is justified by the fact that historically the
first result in this direction was a theorem which describes all
interpolation spaces with respect to the Banach couple $(L^{1},L^{\infty}),$
proved independently by Calder\'{o}n \cite{CAL} and Mityagin \cite{Mit}. In
our terminology, this result is equivalent to the assertion that $
(L^{1},L^{\infty})$ is a uniform Calder\'{o}n-Mityagin couple.

\begin{remark}
\label{rem:CMcoupleAndKmontoneSpaces}

The condition that $\left(X_{0},X_{1}\right)$ is a Calder\'{o}n-Mityagin
couple, obviously implies that every interpolation space with respect to $
\left(X_{0},X_{1}\right)$ is also a $\mathcal{K}$-monotone space.
Furthermore, if $\left(X_{0},X_{1}\right)$ is a uniform Calder\'{o}n-Mityagin
couple, this clearly implies that every interpolation space $X$ with  interpolation constant $C_{1}$ is a uniform $\mathcal{K}$-monotone space with $\mathcal{K}$-monotonicity constant not exceeding $CC_{1}$, where $C$ is the constant appearing in Definition \ref{def:CalderonMityaginCouple}.
\end{remark}

\subsection{Some quasi-Banach sequence spaces and quasi-normed groups}

\label{Prel-sequence-spaces}

As was said above, we will consider, mainly, quasi-Banach spaces which
consist of sequences $x=\left(x_{k}\right)_{k=1}^{\infty}$ of real numbers
with the linear coordinate-wise operations. When $0<p<\infty$, we, as usual,
let $\ell^{p}$ denote the linear space of all sequences for which the
quasi-norm 
\begin{equation*}
\Vert x\Vert_{\ell^{p}}:=\Big(\sum_{k=1}^{\infty}|x_{k}|^{p}\Big)^{1/p} 
\end{equation*}
is finite, and $\ell^{\infty}$ denotes the linear space of all bounded
sequences with the usual norm 
\begin{equation*}
\Vert x\Vert_{\ell^{\infty}}:=\sup_{k=1,2,\dots}|x_{k}|. 
\end{equation*}
For every $0<p<\infty $ and all $x,y\in \ell ^{p}$ we have 
\begin{equation*}
\Vert x+y\Vert _{\ell ^{p}}\leq \max (1,2^{(1-p)/p})\left( \Vert x\Vert
_{\ell ^{p}}+\Vert y\Vert _{\ell ^{p}}\right) 
\end{equation*}
(see e.g. \cite[Lemma~3.10.3]{BL76}).

\vskip0.2cm

Obviously, $\left( \ell ^{p},\ell ^{q}\right) $ is a quasi-Banach couple for all $p,q$ with $0<p,q\leq \infty .$ But we also need to deal with a limiting case of such couples.

\vskip0.2cm

Consideration of the limit of $\Vert
x\Vert_{\ell^{p}}^{p}=\sum_{k=1}^{\infty}|x_{k}|^{p}$ as $p$ tends to $0$
provides the motivation for defining $\ell^{0}$ to be the set of all
sequences $x=(x_{k})_{k=1}^{\infty}$ that are eventually zero, i.e., those
that satisfy 
\begin{equation}
\Vert x\Vert_{\ell^{0}}:=\mathrm{card}(\mathrm{supp}\,x)<\infty,
\label{eq:QuasiNormEllZero}
\end{equation}
where $\mathrm{supp}\,x:=\{k\in\mathbb{N}:\,x_{k}\neq0\}$. Observe that $
\ell^{0}$ is a linear space with respect to the usual coordinate-wise
operations and hence we can consider linear operators defined on $\ell^{0}$.
However, in contrast to the case of $\ell^{p}$ for every $p>0$, $\ell^{0}$
is not a quasi-Banach space, but rather a \textit{quasi-normed group} as
defined in \cite{PS} (see also \cite[\S\,3.10]{BL76}). The functional $\Vert\cdot\Vert_{\ell^{0}}$, although
it is sub-additive, does not have the homogeneity property required for a quasi-norm of a linear space. Indeed, $\ell^{0}$ is an Abelian group of
sequences, where the group operation is coordinate-wise addition.

\begin{remark}
\label{rem:PeetreSparrTerminology}According to the terminology introduced by
Peetre and Sparr in Definitions 1.1 and 2.2 of \cite[p.~219 and pp.~224-225]{PS}, $\ell^{0},$ when equipped with the functional $\left(\text{\ref{eq:QuasiNormEllZero}}\right)$, is an example of a quasi-normed group, and,
more specifically. it is a $(1,1)$-normed Abelian group, and also a $
\left(1,1\mid0\right)$-normed vector space.
\end{remark}


\vskip0.2cm

The extension of the basic concepts and constructions of the interpolation
theory to the class of quasi-normed Abelian groups was initiated by Peetre
and Sparr in the above mentioned paper \cite{PS} (see also \cite[§\,3.11]
{BL76} and \cite{BK91}). In this case the role of morphisms is played,
instead of bounded linear operators, by bounded homomorphisms. Recall that a
mapping $T:\,X\rightarrow X$ on a group $X$ is called a \textit{homomorphism}
on $X$ if $T(x+y)=Tx+Ty$ for all \thinspace \thinspace $x,y\in X$. As in 
\cite[Definition 1.2, p.~223]{PS}, a homomorphism $T$ on $X$ is called 
\textit{bounded} if 
\begin{equation*}
\Vert T\Vert _{X\rightarrow X}:=\sup\limits_{x\neq 0}{\frac{\Vert Tx\Vert }{
\Vert x\Vert }}<\infty .
\end{equation*}

\vskip0.2cm

Note that $\ell ^{0}$ is complete and is linearly and continuously embedded
into the quasi-Banach space $\ell ^{q}$ for every $0<q\leq \infty $ (the
functional $\Vert x\Vert _{\ell ^{0}}$ generates the discrete topology on $
\ell ^{0}$).

\vskip0.2cm

{We shall adopt the following conventions related to homomorphisms which are
bounded on the couple $(\ell ^{0},\ell^{q})$.}

\begin{definition}
\label{def:Int_ell_0_ell_q} $\left(i\right):$ For each $q$ with $
0<q\leq\infty$ we let $\mathfrak{L}\left(\ell^{0},\ell^{q}\right)$ denote
the set of all bounded linear operators on $\ell^{q}$ whose restrictions to $
\ell^{0}$ are bounded homomorphisms.

$\left(ii\right):$ We let $Int\left(\ell^{0},\ell^{q}\right)$ denote the
class of all quasi-normed Abelian groups $E$ which satisfy the continuous
inclusions $\ell^{0}\subset E\subset\ell^{q}$ and which are also
quasi-Banach spaces with respect to their given group quasi-norms and for
which $T:E\rightarrow E$ is bounded for each $T\in\mathfrak{L}
(\ell^{0},\ell^{q})$.
\end{definition}

$\mathfrak{L}\left( \ell ^{0},\ell ^{q}\right) $ is obviously a linear space
and therefore also an Abelian group. Analogously to the usage for couples of
quasi-Banach spaces we define 
\begin{equation*}
\Vert T\Vert _{\mathfrak{L}(\ell ^{0},\ell ^{q})}:=\max (\Vert T\Vert _{\ell
^{0}\rightarrow \ell ^{0}},\Vert T\Vert _{\ell ^{q}\rightarrow \ell ^{q}})
\end{equation*}
for every $T$ in the set $\mathfrak{L}\left( \ell ^{0},\ell ^{q}\right) .$
Then $T\mapsto \Vert T\Vert _{\mathfrak{L}(\ell ^{0},\ell ^{q})}$ is a group
quasi-norm on this set.

\vskip0.2cm

As shown in Remark 2.2 in \cite{CSZ}, using the proof of Theorem 2.1 of that
paper, if $E\in Int\left( \ell ^{0},\ell ^{q}\right) $, then there is a
constant $C$ such that $\Vert T\Vert _{E\rightarrow E}\leq C$ for every $
T\in \mathfrak{L}(\ell ^{0},\ell ^{q})$ with $\Vert T\Vert _{\mathfrak{L}
(\ell ^{0},\ell ^{q})}\leq 1$.

\vskip0.2cm

We adopt a variant of Definition \ref{def:Orbit} and define the orbit of an
element $x\in \ell ^{q}$ with respect to the couple $(\ell ^{0},\ell ^{q})$
to be the linear space $\mathrm{Orb}(x;\ell ^{0},\ell ^{q})$ of all $y\in
\ell ^{q}$, representable in the form $y=Tx$, where $T$ is a bounded linear
operator in $\ell ^{q}$ and is a bounded homomorphism in $\ell ^{0}$. We can
consider this space as a quasi-normed Abelian group by endowing it with the
group quasi-norm 
\begin{equation*}
\Vert y\Vert _{\mathrm{Orb}\left( x\right) }:=\inf \Vert T\Vert _{\mathfrak{L
}(\ell ^{0},\ell ^{q})},
\end{equation*}
where the infimum is taken over all $T\in \mathfrak{L}(\ell ^{0},\ell ^{q})$
such that $y=Tx$.

\vskip0.2cm

Given any $q\in (0,\infty ]$, suppose that $x=(x_{n})_{n=1}^{\infty }$ is an
arbitrary non-zero element of $\ell ^{0}+\ell ^{q}=\ell ^{q}$ so that $
x_{k}\neq 0$ for at least one $k\in \mathbb{N}$. For that $k$ let $x^{\ast }$
be the obviously continuous linear functional on $\ell ^{q}=\ell ^{0}+\ell
^{q}$ defined by $\left\langle y,x^{\ast }\right\rangle =y_{k}$ for each
element $y=\left( y_{n}\right) _{n-1}^{\infty }\in \ell ^{q}$. Since $
\left\langle x,x^{\ast }\right\rangle \neq 0$, we can reason in the same way
as in \cite[§\,1.6, p.~368]{Ovc84} (see also Section~\ref{Prel-Interpolation}
), and show that $\mathrm{Orb}(x;\ell ^{0},\ell ^{q})$ is an interpolation
quasi-normed group between $\ell ^{0}$ and $\ell ^{q}$.

\vskip0.2cm

Note that an inspection of the proofs related to a description of orbits of
elements in the couples $(\ell ^{0},\ell ^{q})$, $0<q\leq \infty $, in the
papers \cite{Ast-94} and \cite{Ast-20} shows that these are completely
consistent with the above definitions. This fact will allow us to further
apply the results of these papers.

\vskip0.2cm

The space $\ell^{p}$, with $0<p<\infty$ (resp. $\ell^{0}$) is an example of
a symmetric quasi-Banach space (resp. group). Recall that a quasi-Banach sequence space (or group) $E$ is said to be a \textit{quasi-Banach sequence lattice} if from $|y_{k}|\leq|x_{k}|$, $k=1,2,\dots$, and $
(x_{k})\in E$ it follows that $(y_{k})\in E$ and $\Vert(y_{k})\Vert_{E}\leq
\Vert(x_{k})\Vert_{E}$. If additionally $E\subset\ell^{\infty}$ and the
conditions $y_{k}^{\ast}=x_{k}^{\ast}$, $k=1,2,\dots$, $(x_{k})\in E$ imply
that $(y_{k})\in E$ and $\Vert(y_{k})\Vert_{E}=\Vert(x_{k})\Vert_{E}$, then $
E$ is called \textit{symmetric}. If $(u_{k})_{k=1}^{\infty}$ is any bounded
sequence then, in what follows, $(u_{k}^{\ast})_{k=1}^{\infty}$ denotes the
nonincreasing permutation of the sequence $(|u_{k}|)_{k=1}^{\infty}$ defined
by 
\begin{equation*}
u_{k}^{\ast}:=\inf_{\mathrm{card}\,A=k-1}\sup_{i\in\mathbb{N}\setminus
A}|u_{i}|,\;\;k\in\mathbb{N}. 
\end{equation*}

A quasi-Banach sequence lattice $E$ has the \textit{Fatou property} if from $x_{n}\in E$, $n=1,2,\dots$, $\sup_{n=1,2,\dots}\Vert x_{n}\Vert_{E}<\infty$
and $x_{n}\rightarrow x$ coordinate-wise as $n\rightarrow\infty$ it follows
that $x\in E$ and $\Vert x\Vert_{E}\leq\liminf_{n\rightarrow\infty}\Vert
x_{n}\Vert_{E}$.

Recall that for all $0\leq p<r<q\leq\infty$ the space $\ell^{r}$ is an
interpolation space between $\ell^{p}$ and $\ell^{q}$ (see e.g. \cite[
Theorem~7.2.2 and Corollary~7.2.3]{BL76}). Moreover, $\ell^{r}$ may be
obtained by applying the classical real $K$-method to the couple $
(\ell^{p},\ell^{q})$ \cite[Theorem~7.1.7]{BL76}.

\vskip 0.5cm

\subsection{The Holmstedt formula and related ${\mathcal{K}}$-functionals.}

\label{Holmstedt}

Further, we repeatedly use the following well-known result due to Holmstedt 
\cite{Holm-70}, which is referred usually as the Holmstedt formula.

Let $0<p<q<\infty$. Then, there exists a positive constant $C_{p,q}$,
depending only on $p$ and $q$, such that for every $f\in L^{p}+L^{q}$ on an
arbitrary underlying measure space it holds 
\begin{eqnarray}
{\mathcal{K}}\left(t,f;L^{p},L^{q}\right) & \le & \left(\int\nolimits
_{0}^{t^{\alpha}}\left(f^{\ast}\left(s\right)\right)^{p}\,ds\right)^{1/p}+t
\left(\int\nolimits
_{t^{\alpha}}^{\infty}\left(f^{\ast}\left(s\right)\right)^{q}\,ds
\right)^{1/q}  \notag \\
& \le & C_{p,q}{\mathcal{K}}\left(t,f;L^{p},L^{q}\right),\;\;t>0,
\label{Holmstedt1}
\end{eqnarray}
where $f^{\ast}$ is the nonincreasing left-continuous rearrangement of the
function $|f|$ and $\alpha$ is given by the formula $1/\alpha=1/p-1/q$.
Similarly, in the case when $q=\infty$ we have 
\begin{equation}
{\mathcal{K}}\left(t,f;L^{p},L^{\infty}\right)\le\left(\int\nolimits
_{0}^{t^{p}}\left(f^{\ast}\left(s\right)\right)^{p}\,ds\right)^{1/p}\le
C_{p,\infty}{\mathcal{K}}\left(t,f;L^{p},L^{\infty}\right),\;\;t>0.
\label{Holmstedt1a}
\end{equation}

If the underlying measure space is the set of positive integers equipped
with the counting measure, the couple $(L^{p},L^{q})$ can be naturally
identified with the couple $(\ell^{p},\ell^{q})$ and so, setting $\widetilde{
f}:=\sum_{n=1}^{\infty}f_{n}\chi_{[n-1,n)}$ for every sequence $
(f_{n})_{n=1}^{\infty}$, we have 
\begin{equation}
{\mathcal{K}}\left(t,(f_{n});\ell^{p},\ell^{q}\right)={\mathcal{K}}\left(t,
\widetilde{f};L^{p}\left(0,\infty\right),L^{q}\left(0,\infty\right)\right),
\;\;t>0.  \label{k-functional-lp-Lp}
\end{equation}
Therefore, since $(\widetilde{f})^{\ast}=\widetilde{f^{*}}$, from 
\eqref{Holmstedt1} and \eqref{Holmstedt1a} it follows that 
\begin{eqnarray}
{\mathcal{K}}\left(t,(f_{n});\ell^{p},\ell^{q}\right) & \le &
\left(\int\nolimits _{0}^{t^{\alpha}}\left(\widetilde{f}^{\ast}\left(s
\right)\right)^{p}\,ds\right)^{1/p}+t\left(\int\nolimits
_{t^{\alpha}}^{\infty}\left(\widetilde{f}^{\ast}\left(s\right)\right)^{q}
\,ds\right)^{1/q}  \notag \\
& \le & C_{p,q}{\mathcal{K}}\left(t,(f_{n});\ell^{p},\ell^{q}\right),\;\;t>0,
\label{Holmstedt2}
\end{eqnarray}
and 
\begin{equation}
{\mathcal{K}}\left(t,(f_{n});\ell^{p},\ell^{\infty}\right)\le\left(\int
\nolimits _{0}^{t^{p}}\left(\widetilde{f}^{\ast}\left(s\right)\right)^{p}
\,ds\right)^{1/p}\le C_{p,\infty}{\mathcal{K}}\left(t,(f_{n});\ell^{p},
\ell^{\infty}\right),\;\;t>0.  \label{Holmstedt2a}
\end{equation}

Let us define now, for every $0<p<q<\infty$, the operators $P_{p}$ and $Q_{q}
$ as follows: if $f\in L^{p}\left(0,\infty\right)+L^{q}\left(0,\infty\right)$
, then 
\begin{eqnarray*}
P_{p}f\left(t\right) & := & \left(\int\nolimits
_{0}^{t}\left(f^{\ast}\left(s\right)\right)^{p}\,ds\right)^{1/p},\;\;t>0, \\
Q_{q}f\left(t\right) & := & \left(\int\nolimits
_{t}^{\infty}\left(f^{\ast}\left(s\right)\right)^{q}\,ds\right)^{1/q},\;
\;t>0.
\end{eqnarray*}
By these notations, inequalities \eqref{Holmstedt2} and \eqref{Holmstedt2a}
can be rewritten as follows: 
\begin{equation}
{\mathcal{K}}\left(t,(f_{n});\ell^{p},\ell^{q}\right)\leq P_{p}\widetilde{f}
\left(t^{\alpha}\right)+tQ_{q}\widetilde{f}\left(t^{\alpha}\right)\leq
C_{p,q}{\mathcal{K}}\left(t,(f_{n});\ell^{p},\ell^{q}\right),\;\;t>0,
\label{Holmsteds_formula3}
\end{equation}
and 
\begin{equation}
{\mathcal{K}}\left(t,(f_{n});\ell^{p},\ell^{\infty}\right)\leq P_{p}
\widetilde{f}\left(t^{p}\right)\leq C_{p,\infty}{\mathcal{K}}
\left(t,(f_{n});\ell^{p},\ell^{\infty}\right),\;\;t>0.
\label{Holmsteds_formula3a}
\end{equation}

In the sequence case, we define the operators $\mathcal{P}_{p}$ and $
\mathcal{Q}_{q}$ by setting for every $x=(x_{k})_{k=1}^{\infty}$ 
\begin{eqnarray*}
\mathcal{P}_{p}x=((\mathcal{P}_{p}x)_{n}),\;\;(\mathcal{P}_{p}x)_{n} & := &
\left(\sum_{k=1}^{n}\left(x_{k}^{\ast}\right)^{p}\right)^{1/p},\;\;n\in
\mathbb{N}, \\
\mathcal{Q}_{q}x=((\mathcal{Q}_{q}x)_{n}),\;\;(\mathcal{Q}_{q}x)_{n} & := &
\left(\sum_{k=n}^{\infty}\left(x_{k}^{\ast}\right)^{q}\right)^{1/q},\;\;n\in
\mathbb{N}.
\end{eqnarray*}

Clearly, for all $x=(x_{k})_{k=1}^{\infty }\in \ell ^{q}$ and $n\in \mathbb{N
}$ we have 
\begin{equation*}
(\mathcal{P}_{p}x)_{n}=P_{p}\widetilde{x}(n)\;\;\mbox{and}\;\;(\mathcal{Q}
_{q}x)_{n}=Q_{q}\widetilde{x}(n).
\end{equation*}
Consequently, inequalities \eqref{Holmsteds_formula3} and 
\eqref{Holmsteds_formula3a} imply for all $n\in \mathbb{N}$
\begin{equation}
{\mathcal{K}}\left( n^{1/\alpha },x;\ell ^{p},\ell ^{q}\right) \leq (
\mathcal{P}_{p}x)_{n}+n^{1/\alpha }(\mathcal{Q}_{q}x)_{n}\leq C_{p,q}{
\mathcal{K}}\left( n^{1/\alpha },x;\ell ^{p},\ell ^{q}\right)   \label{Holmsteds_formula4}
\end{equation}
and 
\begin{equation}
{\mathcal{K}}\left( n^{1/p},x;\ell ^{p},\ell ^{\infty }\right) \leq (
\mathcal{P}_{p}x)_{n}\leq C_{p,\infty }{\mathcal{K}}\left( n^{1/p},x;\ell
^{p},\ell ^{\infty }\right).
\label{Holmsteds_formula4a}
\end{equation}

Let $(X_{0},X_{1})$ be a compatible pair of quasi-normed groups $(X_{0},X_{1})$. We introduce the approximation ${\mathcal{E}}$
-functional by 
\begin{equation*}
{\mathcal{E}}(t,x;X_{0},X_{1}):=\inf \{\Vert x-x_{0}\Vert
_{X_{1}}:\,x_{0}\in X_{0},\Vert x_{0}\Vert _{X_{0}}\leq t\},\quad x\in
X_{0}+X_{1},\ t>0
\end{equation*}
(cf. \cite[Chapter~7]{BL76}). Clearly, the mapping $t\mapsto {\mathcal{E}}
(t,x;X_{0},X_{1})$ is a decreasing function on $(0,\infty )$. 
There is the following connection between the ${\mathcal{E}}-$ and ${
\mathcal{K}}-$functionals (see \cite[§\,7.1]{BL76}): 
\begin{equation}
{\mathcal{K}}(t,x;X_{0},X_{1})=\inf_{s>0}\left(s+t{\mathcal{E}}
(s,x;X_{0},X_{1})\right),\;\;t>0.  \label{EQ13d}
\end{equation}
On the other hand, it is known (see e.g. \cite[Lemma~7.1.3]{BL76}
) that for arbitrary $x\in X_{0}+X_{1}$ we have 
\begin{equation*}
\sup_{s>0}s^{-1}({\mathcal{K}}(s,x;X_{0},X_{1})-t)={\mathcal{E}}
^{*}(t,x;X_{0},X_{1}),\;\;t>0, 
\end{equation*}
where ${\mathcal{E}}^{*}(t,x;X_{0},X_{1})$ is the greatest convex minorant of ${\mathcal{E}}(t,x;X_{0},X_{1})$, and also that for each $\gamma\in(0,1)$ 
\begin{equation*}
{\mathcal{E}}^{*}(t,x;X_{0},X_{1})\le{\mathcal{E}}(t,x;X_{0},X_{1})\le(1-
\gamma)^{-1}{\mathcal{E}}^{*}(\gamma t,x;X_{0},X_{1}),\;\;t>0. 
\end{equation*}
Assuming now that ${\mathcal{K}}(t,y;X_{0},X_{1})\leq C{\mathcal{K}}(t,x;X_{0},X_{1})$ for all $t>0$ 
and applying the last inequalities for $\gamma =1/2$, we get 
\begin{eqnarray*}
{\mathcal{E}}(2t,y;X_{0},X_{1}) &\leq &2{\mathcal{E}}^{\ast
}(t,y;X_{0},X_{1})=2\sup_{s>0}s^{-1}({\mathcal{K}}(s,y;X_{0},X_{1})-t) \\
&\leq &2C\sup_{s>0}s^{-1}({\mathcal{K}}(s,x;X_{0},X_{1})-t/C)=2C{\mathcal{E}}
^{\ast }(t/C,x;X_{0},X_{1}) \\
&\leq &2C{\mathcal{E}}(t/C,x;X_{0},X_{1}),\;\;t>0.
\end{eqnarray*}
As a result, by \eqref{EQ13d} and the latter inequality, we
arrive at the following useful implications: 
\begin{equation}
{\mathcal{E}}(t,y)\leq {\mathcal{E}}(t,x),\;t>0\;\;\Longrightarrow {\mathcal{
K}}(t,y)\leq {\mathcal{K}}(t,x),\;t>0,  \label{impl1}
\end{equation}
and for every $C>0$ 
\begin{equation}
{\mathcal{K}}(t,y)\leq C{\mathcal{K}}(t,x),\;t>0\;\;\Longrightarrow {
\mathcal{E}}(t,y)\leq 2C{\mathcal{E}}(t/(2C),x),\;t>0,  \label{impl2}
\end{equation}
where ${\mathcal{E}}(t,z):={\mathcal{E}}(t,z;X_{0},X_{1})$ and ${\mathcal{K}}
(t,z):={\mathcal{K}}(t,z;X_{0},X_{1})$, $z\in X_{0}+X_{1}$. 

Further, we will apply the above implications to the couples $(\ell^{0},\ell^{q})$, $0<q\le\infty$, and $(L^{0},L^{\infty})$ of (equivalence classes of) measurable
functions on the semi-axis $(0,\infty)$ with the Lebesgue measure $m$. Here, 
$L^{0}=L^{0}(0,\infty)$ is the group (with respect to the usual addition) of
all measurable functions on $(0,\infty)$ with supports of finite measure,
equipped by the quasi-norm 
\begin{equation*}
\Vert f\Vert_{L^{0}}:=m\{t>0:\,f(t)\neq0\}. 
\end{equation*}
One can easily check that, for any $x=(x_{k})_{k=1}^{\infty}\in\ell^{q}$ and all $t\geq0$, we have 
\begin{equation}
{\mathcal{E}}(t,x;\ell^{0},\ell^{q})=\left\{ 
\begin{array}{c}
(\mathcal{Q}x)_{[t]+1}=\Big(\sum_{k=[t]+1}^{\infty}(x_{k}^{\ast})^{q}\Big)
^{1/q}\;\;\mbox{if}\;\;q<\infty \\ 
x_{[t]+1}^{\ast}\;\;\mbox{if}\;\;q=\infty,
\end{array}
\right.  \label{EQ13ddd}
\end{equation}
while for every $f\in L^{0}+L^{\infty}$ and all $t>0$ 
\begin{equation}
{\mathcal{E}}(t,f;L^{0},L^{\infty})=f^{\ast}(t).  \label{eq1b}
\end{equation}

We will use the standard (quasi-)Banach space notation (see e.g. \cite
{LT79-I} and \cite{LT79-II}). In particular, throughout the paper, by $e_{n}$
, $n\in\mathbb{N}$, we denote the vectors of the standard basis in sequence
spaces, and for every sequences $x=\left(x_{n}\right)_{n=1}^{\infty}$, $
y=\left(y_{n}\right)_{n=1}^{\infty}$ we set 
\begin{equation*}
\left\langle x,y\right\rangle :=\sum_{n=1}^{\infty}x_{n}y_{n}\;\;(\mbox{if
the series converges}). 
\end{equation*}
By $[t]$ we denote the integer part of a number $t\in\mathbb{R}$ and by $
\chi_{A}$ the characteristic function of a set $A\subset\mathbb{R}$. 
In what follows, $C$, $c$ etc. denote constants whose value may change from
line to line or even within lines.

\vskip0.2cm

\section{Auxiliary results}

In this section we provide a self-contained presentation of some simple and
apparently well-known facts.

\subsection{An extension theorem for operators bounded on $\ell^{p}$-spaces, 
$0<p<1$}

The purpose of this subsection is to give a detailed account related to an
extension (or extrapolation) theorem for linear operators bounded on $
\ell^{p}$, where $0<p<1.$ This is one of the manifestations of the general
principle, which allows to extend linear mappings bounded on quasi-Banach
spaces to their Banach linear hulls (see, for instance, \cite{KPR84} and
also \cite{Pee74}, where a connection of these constructions with the
interpolation theory is clarified).

\begin{theorem}
\label{Th-bounded-lp} Let $S:\,\ell^{q}\rightarrow\ell^{q}$ be a linear map,
where $0\le q<1.$ Then for every $r$, with $q<r\leq1$, there is a linear
extension $R:\ell^{r}\rightarrow\ell^{r}$ of $S$ such that $\left\Vert
R\right\Vert _{\ell^{r}\to\ell^{r}}\leq\left\Vert S\right\Vert
_{\ell^{q}\to\ell^{q}}$.
\end{theorem}

We begin with proving an auxiliary result, where the following notation will
be used. Let $T:\,\ell^{0}\rightarrow\ell^{\infty}$ be a bounded linear map.
Then $T$ can be identified with an infinite matrix $\left\{ t_{j,k}\right\}
_{j,k=1}^{\infty}$, where $t_{j,k}=\left\langle
e_{j},T\left(e_{k}\right)\right\rangle .$ If $x=\left(x_{n}\right)_{n=1}^{
\infty}\in\ell^{0}$ then $Tx=y$, $y=\left(y_{j}\right)_{j=1}^{\infty}$ is
defined by the finite sum 
\begin{equation*}
y_{j}=\sum_{k=1}^{\infty}t_{j,k}x_{k}. 
\end{equation*}
For an arbitrary $0<q\leq\infty$ let $\Omega_{q}$ denote the space of all
linear maps $T:\ell^{0}\rightarrow\ell^{\infty}$ such that the quantity 
\begin{equation*}
\Theta_{q}\left(T\right):=\sup\left\{ \left\Vert Tx\right\Vert
_{\ell^{q}}:\,x\in\ell^{0},\left\Vert x\right\Vert _{\ell^{q}}=1\right\} 
\end{equation*}
is finite. Clearly, if $T\in\Omega_{q}$ with the matrix $\left(t_{j,k}
\right)_{j,k=1}^{\infty}$, then for each positive integer $k$ the sequence $
t_{k}:=\left(t_{j,k}\right)_{j=1}^{\infty}=T\left(e_{k}\right)$ belongs to $
\ell^{q}$ and moreover 
\begin{equation}
\left\Vert t_{k}\right\Vert _{\ell^{q}}=\left\Vert
T\left(e_{k}\right)\right\Vert
_{\ell^{q}}\leq\Theta_{q}\left(T\right),\;\;k=1,2,\dots.  \label{eqaux1}
\end{equation}
Hence, we see that the condition 
\begin{equation}
\sup_{k=1,2,\dots}\left\Vert t_{k}\right\Vert _{\ell^{q}}<\infty
\label{lp_necc_cond}
\end{equation}
is necessary for $T\in\Omega_{q}.$ Furthermore, we have

\begin{lemma}
\label{Lemma_bounded_lp} Let $T:\ell^{0}\rightarrow\ell^{\infty}$ be a
linear map with the matrix $\left(t_{j,k}\right)_{j,k=1}^{\infty}.$ Let $
t_{k}=\left(t_{j,k}\right)_{j=1}^{\infty},k\in\mathbb{N}.$ Then, if $0<q\leq1
$ we have

$\left(i\right)T\in\Omega_{q}\iff\sup_{k}\left\Vert t_{k}\right\Vert
_{\ell^{q}}<\infty$ and 
\begin{equation}
\Theta_{q}\left(T\right)=\sup_{k=1,2,\dots}\left\Vert t_{k}\right\Vert
_{\ell^{q}};  \label{lp_norm_condition}
\end{equation}

$\left(ii\right)$ if $T\in\Omega_{q}$ then there exists an extension $
\widetilde{T}$ of $T$ to $\ell^{q}$ with 
\begin{equation}
\Vert\widetilde{T}\Vert_{\ell^{q}\rightarrow\ell^{q}}=\sup_{k=1,2,\dots}
\left\Vert t_{k}\right\Vert _{\ell^{q}}.  \label{lp_ext_estimate}
\end{equation}
\end{lemma}

\begin{proof}
$\left(i\right)$. If $T\in\Omega_{q}$ then the reasoning preceding to the
lemma implies condition \eqref{lp_necc_cond}.

Conversely, assume that we have \eqref{lp_necc_cond}. Then, for each $
x=\left(x_{n}\right)_{n=1}^{\infty}\in\ell^{0}$, denoting $Tx=y$, $
y=\left(y_{j}\right)_{j=1}^{\infty}$, and taking into account that $0<q\leq1$
, we have 
\begin{equation*}
\left\vert y_{j}\right\vert =\left\vert
\sum_{k=1}^{\infty}t_{j,k}x_{k}\right\vert \leq\sum_{k=1}^{\infty}\left\vert
t_{j,k}\right\vert \left\vert x_{k}\right\vert
\leq\left(\sum_{k=1}^{\infty}\left\vert t_{j,k}\right\vert ^{q}\left\vert
x_{k}\right\vert ^{q}\right)^{1/q}. 
\end{equation*}
Hence, 
\begin{eqnarray}
\left\Vert Tx\right\Vert _{\ell^{q}}^{q} & = & \sum_{j=1}^{\infty}\left\vert
y_{j}\right\vert ^{q}\leq\sum_{j=1}^{\infty}\sum_{k=1}^{\infty}\left\vert
t_{j,k}\right\vert ^{q}\left\vert x_{k}\right\vert ^{q}  \notag \\
& = & \sum_{k=1}^{\infty}\left(\sum_{j=1}^{\infty}\left\vert
t_{j,k}\right\vert ^{q}\right)\left\vert x_{k}\right\vert
^{q}=\sum_{k=1}^{\infty}\left\Vert t_{k}\right\Vert
_{\ell^{q}}^{q}\left\vert x_{k}\right\vert ^{q}\leq  \notag \\
& \leq & \sup_{k=1,2,\dots}\left\Vert t_{k}\right\Vert
_{\ell^{q}}^{q}\left\Vert x\right\Vert _{\ell^{q}}^{q}.
\label{lp_lemma_estimate_g_norm}
\end{eqnarray}
Therefore, $T\in\Omega_{q}$ and $\Theta_{q}\left(T\right)\leq\sup_{k=1,2,
\dots}\left\Vert t_{k}\right\Vert _{\ell^{q}}$. Moreover, combining this
with inequality \eqref{eqaux1} and taking into account that $\ell^{0}$ is
dense in $\ell^{q}$, we get \eqref{lp_norm_condition}.

$\left(ii\right).$ Let $T\in\Omega_{q}$. Since $\ell^{0}$ is dense in $
\ell^{q}$, it follows from inequality \eqref{lp_lemma_estimate_g_norm} that
we can define the linear extension $\widetilde{T}:\ell^{q}\rightarrow\ell^{q}
$ of $T$ by $\widetilde{T}x=y$, where $x=\left(x_{j}\right)_{j=1}^{\infty}$, 
$y_{j}=\sum_{k=1}^{\infty}t_{j,k}x_{k}.$ Since $\Vert\widetilde{T}
\Vert_{\ell^{q}\rightarrow\ell^{q}}=\Theta_{q}\left(T\right)$, then in view
of \eqref{lp_norm_condition}, formula \eqref{lp_ext_estimate} is also
verified.
\end{proof}

\begin{proof}[Proof of Theorem~\protect\ref{Th-bounded-lp}]
Let $T$ be the restriction of the given operator $S$ to the space $\ell^{0}.$
Then $T\in\Omega_{q}$ and $\Theta_{q}\left(T\right)=\left\Vert S\right\Vert
_{\ell^{q}\to\ell^{q}}.$ It is obvious that the extension $\widetilde{T}$ of 
$T$ to $\ell^{q}$ defined in Lemma \ref{Lemma_bounded_lp} is equal to $S.$
Now, if $\left(t_{j,k}\right)_{j,k=1}^{\infty}$ is the matrix associated
with $T$ and $t_{k}=\left(t_{j,k}\right)_{j=1}^{\infty}$, the embedding $
\ell^{q}\overset{1}{\subset}\ell^{r}$, for $q<r$, implies that $\left\Vert
t_{k}\right\Vert _{\ell^{r}}\leq\left\Vert t_{k}\right\Vert _{\ell^{q}}$, $
k=1,2,\dots,$ and hence $T\in\Omega_{r}$ with $\Theta_{r}\left(T\right)\leq
\Theta_{q}\left(T\right).$ Therefore, by Lemma \ref{Lemma_bounded_lp}, there
exists an extension $R$ of $T$ defined on $\ell^{r}$. Clearly, the
restriction of $R$ to $\ell^{q}$ equals $S$ and 
\begin{equation*}
\left\Vert R\right\Vert
_{\ell^{r}\to\ell^{r}}=\Theta_{r}\left(T\right)\leq\Theta_{q}\left(T\right)=
\left\Vert S\right\Vert _{\ell^{q}\to\ell^{q}}, 
\end{equation*}
which completes the proof.
\end{proof}

\begin{corollary}
\label{triv cor} If $0\leq p<q<1$, then we have 
\begin{equation*}
Int\left(\ell^{p},\ell^{q}\right)=\{X\subseteq\ell^{q}:\,X\in
Int\left(\ell^{p},\ell^{1}\right)\}. 
\end{equation*}
\end{corollary}

\begin{proof}
Let $X\in Int\left(\ell^{p},\ell^{q}\right)$. Then, $X\subseteq\ell^{q}$.
Moreover, since $\ell^{q}\in Int\left(\ell^{p},\ell^{1}\right)$, then for
each $S\in\mathfrak{L}(\ell^{p},\ell^{1})$ we have $S:\ell^{q}\rightarrow
\ell^{q}$. Thus, $S:\,X\to X$ and so $X\in Int\left(\ell^{p},\ell^{1}\right).
$

For the converse, assume that $X\subseteq\ell^{q}$ and $X\in
Int\left(\ell^{p},\ell^{1}\right)$. Then, $\ell^{p}\subseteq
X\subseteq\ell^{q}$. Furthermore, if $S\in\mathfrak{L}(\ell^{p},\ell^{q})$,
an application of Theorem~\ref{Th-bounded-lp} gives us a linear extension $
R:\,\ell^{1}\rightarrow\ell^{1}$ of $S$. Since $R:\,X\rightarrow X$, then $
S=R_{|{\ell^{q}}}:\,X\to X.$ As a result, we conclude that $X\in
Int\left(\ell^{p},\ell^{q}\right).$
\end{proof}


\subsection{An interpolation property of symmetric quasi-Banach sequence
spaces}

\label{An interpolation property}

It is well known that there are symmetric Banach function (resp. sequence)
spaces which are \textit{not} interpolation spaces with respect to the
couple $(L^{1}(0,\infty),L^{\infty}(0,\infty))$ (resp. $(\ell^{1},\ell^{
\infty})$) (see e.g. \cite{Russu}, \cite[Theorem~II.5.5]{KPS82}, \cite[
Example 2.a.11]{LT79-II}). In contrast to that, all symmetric quasi-Banach
function (resp. sequence) spaces are interpolation spaces with respect to
the couple $(L^{0}(0,\infty),L^{\infty}(0,\infty))$ (resp. $
(\ell^{0},\ell^{\infty})$) \cite{Ast-94,HM-90}. Moreover, it turns out that each symmetric quasi-Banach sequence space is an interpolation space between $\ell^{p}$ and $\ell^{\infty}$ for some appropriate $p>0$.

\vskip0.5cm

For every $n\in\mathbb{N}$ we define the dilation operator $
D_{n}:\,\ell^{\infty}\rightarrow\ell^{\infty}$ by 
\begin{equation}
D_{n}:\,(x_{k})_{k=1}^{\infty}\mapsto\left(x_{[(1+k)/n]}\right)_{k=1}^{\infty}.
\label{extra8}
\end{equation}

Let $E$ be a symmetric quasi-normed sequence space. By the Aoki-Rolewicz
theorem (see e.g. \cite[Lemma~3.10.1]{BL76}), we can define a subadditive
functional on $E$, which is equivalent to the functional $\left\Vert
\cdot\right\Vert _{E}^{\sigma_{E}}$ for some $\sigma_{E}>0$ (which is called often the Aoki-Rolewicz index). Therefore, since $E$ is symmetric, by the definition of $D_{n}
$, one can easily deduce that $\Vert D_{n}\Vert_{E\rightarrow E}\leq
Cn^{1/\sigma_{E}}$, $n=1,2,\dots$. Hence, the upper Boyd index $q_{E}$ of $E$, defined by 
\begin{equation*}
q_{E}:=\lim_{n\rightarrow\infty}\frac{\log\Vert D_{n}\Vert_{E\rightarrow E}}{
\log n}, 
\end{equation*}
does not exceed $1/\sigma_{E}$ and so it is finite.

\begin{proposition}
\label{L3} For every symmetric quasi-Banach sequence space $E$ there exists $
p>0$ such that $E\in Int(\ell^{p},\ell^{\infty})$.
\end{proposition}

In the proof of this result we will use

\begin{lemma}
(\cite[Theorem~1]{Mont-94} or \cite[Proposition 5.7]{CSZ}). \label{L2} Let $E
$ be a symmetric quasi-Banach sequence space and let $p\in(0,\sigma_{E})$.
Then, the operator 
\begin{equation*}
x=(x_{k})_{k=1}^{\infty}\mapsto\left(\left(\frac{1}{n}\sum_{k=1}^{n}(x_{k}^{
\ast})^{p}\right)^{1/p}\right)_{n=1}^{\infty} 
\end{equation*}
is bounded in $E$.
\end{lemma}

\begin{proof}[Proof of Proposition~\protect\ref{L3}]
By \cite[Theorem~3]{Cwikel1}, it suffices to prove that there is a constant $
C>$ such that for any $x\in E$ and $y\in\ell^{\infty}$ satisfying 
\begin{equation*}
\sum_{k=1}^{n}(y_{k}^{\ast})^{p}\leq\sum_{k=1}^{n}(x_{k}^{\ast})^{p},\;
\;n=1,2,\dots, 
\end{equation*}
we have $y\in E$ and $\Vert y\Vert_{E}\leq C\Vert x\Vert_{E}$.

Since 
\begin{equation*}
y_{n}^{\ast}\leq\Big(\frac{1}{n}\sum_{k=1}^{n}(y_{k}^{\ast})^{p}\Big)
^{1/p},\;\;n=1,2,\dots, 
\end{equation*}
then by the preceding lemma we have 
\begin{equation*}
\Vert y\Vert_{E}\leq\left\Vert \left(\Big(\frac{1}{n}\sum_{k=1}^{n}(y_{k}^{
\ast})^{p}\Big)^{1/p}\right)_{n}\right\Vert _{E}\leq\left\Vert \left(\Big(
\frac{1}{n}\sum_{k=1}^{n}(x_{k}^{\ast})^{p}\Big)^{1/p}\right)_{n}\right\Vert
_{E}\leq C\Vert x\Vert_{E}. 
\end{equation*}
\end{proof}

\vskip1cm 

\section{Arazy-Cwikel type properties for the scale of $\ell^{p}$-spaces, $
0\leq p\leq\infty$.}

\label{Section-AC} In \cite{Arazy-Cwikel}, Arazy and Cwikel have proved that
for all $1\leq p<q\leq\infty$ and for each underlying measure space 
\begin{equation*}
Int\left(L^{p},L^{q}\right)=Int\left(L^{1},L^{q}\right)\cap
Int\left(L^{p},L^{\infty}\right). 
\end{equation*}

Recently, a similar result has been obtained in \cite{CSZ} for function
quasi-Banach spaces on $(0,\infty)$ with the Lebesgue measure; namely, it was proved there that for all $0<p<q<\infty$ 
\begin{equation*}
Int\left(L^{p},L^{q}\right)=Int\left(L^{0},L^{q}\right)\cap
Int\left(L^{p},L^{\infty}\right). 
\end{equation*}

In this section, we extend the above Arazy-Cwikel theorem in the sequence
space setting.

\begin{theorem}
\label{int spaces} Let $0\leq s<p<q<r\leq\infty.$ Then, we have 
\begin{equation*}
Int\left(\ell^{p},\ell^{q}\right)=Int\left(\ell^{s},\ell^{q}\right)\cap
Int\left(\ell^{p},\ell^{r}\right). 
\end{equation*}
\end{theorem}

Further, in Section~\ref{fails CM}, we prove that the couple $
\left(\ell^{p},\ell^{q}\right)$ fails to have the Calder\'{o}n-Mityagin
property if $0\leq p<q<1$. Thus, according to Theorem \ref{int spaces},
Arazy-Cwikel type results do not imply, in general, that the corresponding
couples possess necessarily the Calder\'{o}n-Mityagin property.

The proof of Theorem \ref{int spaces} will be based on using a series of
auxiliary results, the first of them is well-known. 

\begin{proposition}
\cite[Theorem~3]{Cwikel1} \label{p1} Let $0<p<\infty$, and let $
x=\left(x_{k}\right)_{k=1}^{\infty}\in\ell^{\infty}$, $y=\left(y_{k}
\right)_{k=1}^{\infty}\in\ell^{\infty}$ be two sequences satisfying 
\begin{equation*}
\sum_{k=1}^{n}\left(y_{k}^{\ast}\right)^{p}\leq\sum_{k=1}^{n}\left(x_{k}^{
\ast}\right)^{p},\;\;n=1,2,\dots 
\end{equation*}
Then there exists a linear operator $T:\,\ell^{\infty}\rightarrow\ell^{
\infty}$ such that $\left\Vert T\right\Vert
_{\ell^{p}\rightarrow\ell^{p}}\leq8^{1/p}$, $\left\Vert T\right\Vert
_{\ell^{\infty}\rightarrow\ell^{\infty}}\leq2^{1/p}$ and $Tx=y$.
\end{proposition}

\begin{proposition}
\label{p2} Let $1\le q<\infty$, $C\ge1$, and let $x=\left(x_{k}
\right)_{k=1}^{\infty}\in\ell^{q}$, $y=\left(y_{k}\right)_{k=1}^{\infty}\in
\ell^{q}$ be two sequences such that 
\begin{equation*}
\sum_{k=n}^{\infty}\left(y_{k}^{*}\right)^{q}\le
C\sum_{k=[(n-1)/C]+1}^{\infty}\left(x_{k}^{*}\right)^{q},\;\;n=1,2,\dots 
\end{equation*}
Then there exists a linear operator $S:\,\ell^{q}\rightarrow\ell^{q}$ such that $
\left\Vert S\right\Vert _{\ell^{0}\to\ell^{0}}\le9(C+1)$, $\left\Vert
S\right\Vert _{\ell^{q}\to\ell^{q}}\le6(C+1)$ and $Sx=y$.
\end{proposition}

\begin{proof}
Without loss of generality, we can assume that $x=\left(x_{n}\right)_{n=1}^{
\infty}$ and $y=\left(y_{n}\right)_{n=1}^{\infty}$ are nonnegative and nonincreasing.

By \cite[Theorem~1; see also its proof]{Ast-20}, we can find a positive
linear operator $Q:\,\ell^{1}\rightarrow\ell^{1}$ such that $\left\Vert
Q\right\Vert _{\ell^{0}\to\ell^{0}}\le9(C+1)$, $\left\Vert Q\right\Vert
_{\ell^{1}\to\ell^{1}}\le6(C+1)$ and $Q((x_{k}^{q})_{k=1}^{
\infty})=((y_{k}^{q})_{k=1}^{\infty})$. Let us define the mapping $Q^{\prime}
$ by 
\begin{equation*}
Q^{\prime}((z_{k})_{k=1}^{\infty}):=\left(Q((|z_{k}|^{q})_{k=1}^{\infty})
\right)^{1/q}. 
\end{equation*}
Then, $Q^{\prime}$ is a subadditive and positively homogeneous operator
bounded on $\ell^{q}$ such that $\left\Vert Q^{\prime}\right\Vert
_{\ell^{0}\to\ell^{0}}\le\left\Vert Q\right\Vert _{\ell^{0}\to\ell^{0}}$, $
\left\Vert Q^{\prime}\right\Vert _{\ell^{q}\to\ell^{q}}\le\left\Vert
Q\right\Vert _{\ell^{1}\to\ell^{1}}$ and $Q^{\prime}x=y$.

Now, let us define the linear operator $S^{\prime}$ on the one-dimensional
subspace $H_{x}:=\{\alpha x,\,\alpha\in\mathbb{R}\}$ in $\ell^{q}$ by $
S^{\prime}(\alpha x):=\alpha y$, $\alpha\in\mathbb{R}$. Then, $
S^{\prime}z\le Q^{\prime}z$, $z\in H_{x}$, and $S^{\prime}x=y$. Hence, by
the Hahn-Banach-Kantorovich theorem (see e.g. \cite[p.~120]{Rubinov}), there exists a linear extension $S$ of $S^{\prime}$ to the whole of $\ell^{q}$
such that $Sz\le Q^{\prime}z$ for all $z\in\ell^{q}$. Since the operator $S$
satisfies all the requirements, the proof is completed.
\end{proof}

\begin{remark}
Alternatively, instead of Theorem~1 from \cite{Ast-20}, in the proof of
Proposition \ref{p2} one can apply Lemma~5.2 from \cite{CSZ} together with using the dilation operators (see \eqref{extra8}).
\end{remark}

\begin{theorem}
\label{p3} Suppose $0<p<q<\infty$ and $q\ge 1$. Let $x=\left(x_{k}\right)_{k=1}^{\infty}
\in\ell^{q}$ and $y=\left(y_{k}\right)_{k=1}^{\infty}\in\ell^{q}$ be two
sequences such that 
\begin{equation*}
\left(\sum_{k=1}^{n}\left(y_{k}^{\ast}\right)^{p}\right)^{1/p}+n^{1/\alpha}
\left(\sum_{k=n}^{\infty}\left(y_{k}^{\ast}\right)^{q}\right)^{1/q}\leq
\left(\sum_{k=1}^{n}\left(x_{k}^{\ast}\right)^{p}\right)^{1/p}+n^{1/\alpha}
\left(\sum_{k=n}^{\infty}\left(x_{k}^{\ast}\right)^{q}\right)^{1/q},\;\;n\in
\mathbb{N}, 
\end{equation*}
where $1/\alpha=1/p-1/q$.

Then, we can find linear operators $T:\,\ell^{\infty}\rightarrow\ell^{\infty}
$ and $S:\,\ell^{q}\rightarrow\ell^{q}$ such that $\left\Vert T\right\Vert _{
\mathfrak{L}(\ell^{p},\ell^{\infty})}\leq8^{1/p}$, $\left\Vert S\right\Vert
_{\mathfrak{L}(\ell^{0},\ell^{q})}\leq 18$, with $y=Tx+Sx$.
\end{theorem}

\begin{proof}
As above, we can (and will) assume that $x=\left(x_{n}\right)_{n=1}^{\infty}$
and $y=\left(y_{n}\right)_{n=1}^{\infty}$ are nonnegative and nonincreasing sequences. The proof below will be modelled on the arguments used in \cite
{Arazy-Cwikel}.

Let us define 
\begin{equation*}
A\left(n\right):=\sum_{k=1}^{n}\left(x_{k}^{p}-y_{k}^{p}\right),\;n\in
\mathbb{N},\;\;\mbox{and}\;\;A:=\left\{ n\in\mathbb{N}:\,A\left(n\right)
\geq0\right\} , 
\end{equation*}
\begin{equation*}
B\left(n\right):=\sum_{k=n}^{\infty}\left(x_{k}^{q}-y_{k}^{q}\right),\;n\in
\mathbb{N},\;\;\mbox{and}\;\;B:=\left\{ n\in\mathbb{N}:\,B\left(n\right)
\geq0\right\} . 
\end{equation*}
Then, from the assumption of the theorem it follows that $A\cup B=\mathbb{N}$.

Observe that in the case when $A=\mathbb{N}$ it follows 
\begin{equation*}
\sum_{k=1}^{n}y_{k}^{p}\leq\sum_{k=1}^{n}x_{k}^{p},\;\;n\in\mathbb{N}, 
\end{equation*}
and hence Proposition \ref{p1} implies that $y=Tx$ for some operator $
T:\ell^{\infty}\to\ell^{\infty}$ bounded in the couple $\left(\ell^{p},
\ell^{\infty}\right).$ Similarly, if $B=\mathbb{N}$ then 
\begin{equation*}
\sum_{k=n}^{\infty}y_{k}^{q}\leq\sum_{k=n}^{\infty}x_{k}^{q},\;\;n\in\mathbb{
N}. 
\end{equation*}
Consequently, since $q\ge 1$, by Proposition \ref{p2}, $y=Sx$ for some $S:\,\ell^{q}\to\ell^{q}$ bounded in the couple $\left(\ell^{0},\ell^{q}\right)$. So, in these cases the desired result follows.

Assume now that neither $A=\mathbb{N}$ nor $B=\mathbb{N}$. We represent $A$
as the union of successive maximal pairwise disjoint intervals of positive
integers, i.e., $A=\cup_{i\in I_{1}}A_{i}$, where $A_{i}=\left[n_{i},m_{i}
\right]$, $m_{i}+1<n_{i+1}$. Let $B=\cup_{i\in I_{2}}B_{i}$ be the
corresponding union for $B.$ These collections of intervals may be finite or
infinite.

Let $1\in A$, i.e., $A_{1}=[1,m_{1}]$. Then, $A\left(\min\left(l,m_{1}
\right)\right)\ge0$ for every $l\in\mathbb{N}$, and hence we have 
\begin{eqnarray*}
\sum_{j=1}^{l}\left(\left(\chi_{A_{1}}x\right)_{j}^{\ast}\right)^{p} & = &
\sum_{j=1}^{\min\left(l,m_{1}\right)}x_{j}^{p}=A\left(\min\left(l,m_{1}
\right)\right)+\sum_{j=1}^{\min\left(l,m_{1}\right)}y_{j}^{p} \\
& \ge &
\sum_{j=1}^{\min\left(l,m_{1}\right)}y_{j}^{p}=\sum_{j=1}^{l}\left(\left(
\chi_{A_{1}}y\right)_{j}^{\ast}\right)^{p}.
\end{eqnarray*}

Suppose now that $A_{i}=\left[n_{i},m_{i}\right]$, where $n_{i}\ge2$, be a
finite interval. Then, $\min\left(l+n_{i}-1,m_{i}\right)\in A$, $n_{i}-1\in
A^{c}:=\mathbb{N}\setminus A$ for all $l\in\mathbb{N}$. Therefore, $A\left(\min
\left(l+n_{i}-1,m_{i}\right)\right)\geq0$, $A\left(n_{i}-1\right)\leq0$ and 
\begin{eqnarray*}
\sum_{j=1}^{l}\left(\left(\chi_{A_{i}}x\right)_{j}^{\ast}\right)^{p} & = &
\sum_{j=n_{i}}^{l+n_{i}-1}\left(\chi_{A_{i}}x\right)_{j}^{p}=
\sum_{j=n_{i}}^{\min\left(l+n_{i}-1,m_{i}\right)}x_{j}^{p} \\
& = &
A\left(\min\left(l+n_{i}-1,m_{i}\right)\right)-A\left(n_{i}-1\right)+
\sum_{j=n_{i}}^{\min\left(l+n_{i}-1,m_{i}\right)}y_{j}^{p} \\
& \geq &
\sum_{j=n_{i}}^{\min\left(l+n_{i}-1,m_{i}\right)}y_{j}^{p}=\sum_{j=1}^{l}
\left(\left(\chi_{A_{i}}y\right)_{j}^{\ast}\right)^{p}.
\end{eqnarray*}

If finally $A_{i}=\left[n_{i},\infty\right)$ then, for any $l\in\mathbb{N}$,
we have $n_{i}+l-1\in A_{i}$ and $n_{i}-1\in A^{c}.$ Hence, as above, 
\begin{eqnarray*}
\sum_{j=1}^{l}\left(\left(\chi_{A_{i}}x\right)_{j}^{\ast}\right)^{p} & = &
\sum_{j=n_{i}}^{n_{i}+l-1}x_{j}^{p}=A\left(n_{i}+l-1\right)-A\left(n_{i}-1
\right)+\sum_{j=n_{i}}^{n_{i}+l-1}y_{j}^{p} \\
& \geq &
\sum_{j=n_{i}}^{n_{i}+l-1}y_{j}^{p}=\sum_{j=1}^{l}\left(\left(\chi_{A_{i}}y
\right)_{j}^{\ast}\right)^{p}.
\end{eqnarray*}

Thus, by the estimates obtained and Proposition \ref{p1}, for each $i\in
I_{1}$, we can select a linear operator $T_{i}:\,\ell^{\infty}\rightarrow
\ell^{\infty}$, $\left\Vert T_{i}\right\Vert _{\mathfrak{L}
(\ell^{p},\ell^{\infty})}\leq8^{1/p}$, with $T_{i}(\chi_{A_{i}}x)=
\chi_{A_{i}}y$. Now setting $T=\sum_{i\in I_{1}}\chi_{A_{i}}T_{i}\chi_{A_{i}}
$, we see that the operator $T$ is well-defined on $\ell^{\infty}$, 
\begin{equation*}
\left\Vert T\right\Vert _{\mathfrak{L}(\ell^{p},\ell^{\infty})}\leq\sup_{i
\in I_{1}}\left\Vert T_{i}\right\Vert _{\mathfrak{L}(\ell^{p},\ell^{
\infty})}\le8^{1/p} 
\end{equation*}
and $T(\chi_{A}x)=\chi_{A}y$.

Similarly, let $B_{i}=\left[n_{i}^{\prime},m_{i}^{\prime}\right]$ be a
finite interval, and let $l\in\mathbb{N}$ be such that $l\leq
m_{i}^{\prime}-n_{i}^{\prime}+1$. Then, $l+n_{i}^{\prime}-1\in B$ and $
m_{i}^{^{\prime}}+1\in B^{c}$. Consequently, $B(l+n_{i}^{\prime}-1)\geq0$
and $B(m_{i}^{\prime}+1)\leq0$. Hence, 
\begin{eqnarray*}
\sum_{j=l}^{\infty}\left(\left(\chi_{B_{i}}x\right)_{j}^{\ast}\right)^{q} &
= &
\sum_{j=l+n_{i}^{\prime}-1}^{m_{i}^{\prime}}x_{j}^{q}=B(l+n_{i}^{
\prime}-1)-B(m_{i}^{\prime}+1) \\
& + &
\sum_{j=l+n_{i}^{\prime}-1}^{m_{i}^{\prime}}y_{j}^{q}\geq\sum_{j=l+n_{i}^{
\prime}-1}^{m_{i}^{\prime}}y_{j}^{q}=\sum_{j=l}^{\infty}\left(\left(
\chi_{B_{i}}y\right)_{j}^{\ast}\right)^{q}.
\end{eqnarray*}
Observe that the latter inequality holds also for $l>m_{i}^{\prime}-n_{i}^{
\prime}+1$, because in this case its both sides vanish.

In the case when $B_{i}=\left[n_{i}^{\prime},\infty\right)$ for every $l\in
\mathbb{N}$ we get 
\begin{eqnarray*}
\sum_{j=l}^{\infty}\left(\left(\chi_{B_{i}}x\right)_{j}^{\ast}\right)^{q} &
= &
\sum_{j=n_{i}^{\prime}+l-1}^{\infty}x_{j}^{q}=B\left(n_{i}^{\prime}+l-1
\right)+\sum_{j=n_{i}+l-1}^{\infty}y_{j}^{q} \\
& \geq &
\sum_{j=n_{i}+l-1}^{\infty}y_{j}^{q}=\sum_{j=l}^{\infty}\left(\left(
\chi_{B_{i}}y\right)_{j}^{\ast}\right)^{q},
\end{eqnarray*}
because of $n_{i}^{\prime}+l-1\in B_{i}$ for all $l\in\mathbb{N}$.

As a result, by using Proposition \ref{p2}, for every $i\in I_{2}$ we can
find an operator $S_{i}:\,\ell^{q}\rightarrow\ell^{q}$, $\|S_{i}\|_{
\mathfrak{L}(\ell^{0},\ell^{q})}\le 18$, with $S_{i}(\chi_{B_{i}}x)=
\chi_{B_{i}}y.$ Then, the operator $S^{\prime}:=\sum_{i=1}^{\infty}
\chi_{B_{i}}S_{i}\chi_{B_{i}}$ is well-defined on $\ell^{q}$, 
\begin{equation*}
\left\Vert S^{\prime}\right\Vert _{\mathfrak{L}(\ell^{0},\ell^{q})}\leq
\sup_{i\in I_{2}}\left\Vert S_{i}\right\Vert _{\mathfrak{L}
(\ell^{0},\ell^{q})}\le 18 
\end{equation*}
and $S^{\prime}(\chi_{B}x)=\chi_{B}y$. Denoting $S:=\chi_{B\setminus
A}S^{\prime}$, we see that $\left\Vert S\right\Vert _{\mathfrak{L}
(\ell^{0},\ell^{q})}\leq 18$, and $Sx=\chi_{B\setminus A}y$. Since 
\begin{equation*}
y=\chi_{A}y+\chi_{B\setminus A}y=Tx+Sx, 
\end{equation*}
the operators $T:\,\ell^{\infty}\rightarrow\ell^{\infty}$ and $
S:\,\ell^{q}\rightarrow\ell^{q}$ satisfy all the requirements and so the
proof of the theorem is completed.
\end{proof}

It is a classical result of the interpolation theory that the couple $
\left(L^{p},L^{q}\right)$, $1\leq p<q\leq\infty$, has the uniform Calder\'{o}n-Mityagin property with respect to the class of all interpolation Banach spaces (note that it is also a special case of the well-known Sparr theorem, see \cite{SP78}). The preceding results of this section combined
with the observations made in Subsection \ref{Prel-sequence-spaces} imply the following extension of this result to the quasi-Banach case in the sequence space setting.

\begin{corollary}
\label{Ar-Cw} Let $0\le p<q\leq\infty$ and $q\geq1.$ Then $
\left(\ell^{p},\ell^{q}\right)$ is a uniform Calder\'{o}n-Mityagin couple.
\end{corollary}

\begin{proof}
Let $x=\left(x_{k}\right)_{k=1}^{\infty}\in\ell^{q}$ and $
y=\left(y_{k}\right)_{k=1}^{\infty}\in\ell^{q}$ be two nonincreasing and nonnegative sequences such that 
\begin{equation}
{\mathcal{K}}(t,y;\ell^{p},\ell^{q})\le{\mathcal{K}}(t,x;\ell^{p},\ell^{q}),
\;\;t>0.  \label{equ1}
\end{equation}
We consider four cases depending on values of the numbers $p$ and $q$
separately.

Let first $0<p<q=\infty$. Then, by \eqref{Holmsteds_formula4a}, we have 
\begin{equation*}
\sum_{k=1}^{n}\left(y_{k}^{\ast}\right)^{p}\leq
C_{p,\infty}^{p}\sum_{k=1}^{n}\left(x_{k}^{\ast}\right)^{p},\;\;n=1,2,\dots 
\end{equation*}
Hence, from Proposition \ref{p1} it follows the existence of a linear
operator $T:\,\ell^{\infty}\rightarrow\ell^{\infty}$ such that $\left\Vert
T\right\Vert _{\mathfrak{L}(\ell^{p},\ell^{\infty})}\leq8^{1/p}C_{p,\infty}$
and $Tx=y$. Thus, $y\in{\mathrm{Orb}}(x;\ell^{p},\ell^{\infty})$ and $\Vert
y\Vert_{{\mathrm{Orb}}(x)}\leq8^{1/p}C_{p,\infty}$.

Suppose now $p=0$ and $q=\infty$. 
Then from \eqref{impl2}, \eqref{EQ13ddd} and the inequality $[k/2]+1\ge [(k+1)/2]$, $k=1,2,\dots$, it follows that 
\begin{equation*}
y_{k}\le2x_{[k/2]+1}\le 2x_{[(k+1)/2]}=2(D_{2}x)_{k},\;\;k=1,2,\dots, 
\end{equation*}
where $D_{2}$ is the doubling operator (see \eqref{extra8}). For each $u=(u_{k})_{k=1}^{\infty}\in\ell^{\infty}$ we define
the multiplication operator $T$ by 
\begin{equation*}
Tu=\left(u_{k}\cdot\frac{y_{k}}{2(D_{2}x)_{k}}\right)_{k=1}^{\infty}. 
\end{equation*}
Obviously, $\|T\|_{\mathfrak{L}(\ell^{0},\ell^{\infty})}\le1$. Therefore, if 
$S:=2TD_{2}$, then $\|S\|_{\mathfrak{L}(\ell^{0},\ell^{\infty})}\le2$ and $
Sx=y$. Thus, $y\in{\mathrm{Orb}}(x;\ell^{0},\ell^{\infty})$ and $\Vert
y\Vert_{{\mathrm{Orb}}(x)}\leq2$.

Next, if $0=p<q<\infty$, from estimate \eqref{equ1}, implication 
\eqref{impl2} with $X_{0}=\ell^{0}$, $X_{1}=\ell^{q}$ and formula 
\eqref{EQ13ddd} it follows that 
\begin{equation*}
\sum_{k=n}^{\infty}\left(y_{k}^{*}\right)^{q}\le2\sum_{k=[(n-1)/2]+1}^{\infty}\left(x_{k}^{*}\right)^{q},\;\;n=1,2,\dots 
\end{equation*}
Then, by Proposition \ref{p2}, there exists a linear operator $
S:\,\ell^{q}\rightarrow\ell^{q}$ such that $\left\Vert S\right\Vert _{
\mathfrak{L}(\ell^{0},\ell^{q})}\le27$ and $Sx=y$. Therefore, we obtain
again that $y\in{\mathrm{Orb}}(x;\ell^{0},\ell^{q})$ and $\|y\|_{{\mathrm{Orb
}}}\le27$.

Finally, suppose $0<p<q<\infty$. Then, combining \eqref{equ1} with the
Holmstedt formula \eqref{Holmsteds_formula4} and applying Theorem \ref{p3},
we can find linear operators $T:\,\ell^{\infty}\rightarrow\ell^{\infty}$ and 
$S:\,\ell^{q}\rightarrow\ell^{q}$ such that $\left\Vert T\right\Vert _{
\mathfrak{L}(\ell^{p},\ell^{\infty})}\leq8^{1/p}C_{p,q}$ and $\left\Vert
S\right\Vert _{\mathfrak{L}(\ell^{0},\ell^{q})}\leq 18C_{p,q}$, with $y=Tx+Sx$. Since by interpolation $T$ (resp. $S$) is bounded in $\ell^{q}$ (resp. $
\ell^{p}$) and $\Vert T\Vert_{\ell^{q}\rightarrow\ell^{q}}\leq
C_{p,q}^{\prime}$ (resp. $\Vert S\Vert_{\ell^{p}\rightarrow\ell^{p}}\leq
C_{p,q}^{\prime\prime}$), we conclude that $y\in{\mathrm{Orb}}
(x;\ell^{p},\ell^{q})$ and $\Vert y\Vert_{{\mathrm{Orb}}(x)}\leq
C_{p,q}^{\prime}+C_{p,q}^{\prime\prime}$. Thus, the theorem is proved.
\end{proof}

\begin{remark}
According to \cite[Theorem~1.1]{CSZ}, the result, analogous to Theorem \ref{p3}, holds for the couple $\left(L^{p},L^{q}\right)$ of functions defined on the semi-axis $(0,\infty)$ with the Lebesgue measure for all
$0\le p<q\leq\infty$ without any extra conditions imposed on $q$. Consequently, arguing in the same way as in the proof of Corollary \ref{Ar-Cw}, we conclude that the couple $\left(L^{p},L^{q}\right)$ is a uniform Calder\'{o}n-Mityagin couple for all
$0\le p<q\leq\infty$. At the same time, it is worth to note that the condition $q\geq1$ cannot be skipped in Theorem \ref{p3} and Corollary \ref{Ar-Cw}  (see Corollary \ref{th2} below), which shows an essential difference in interpolation properties of quasinormed $L^p$-couples in  function and sequence cases.  
\end{remark}

Applying Corollary \ref{Ar-Cw} and Theorem~\ref{Th-bounded-lp}, we obtain a complete description of orbits of elements in the couple $(\ell^{p},\ell^{q})$ for all nonnegative values $p$ and $q$.

\begin{corollary}
\label{Corr-1} Let $0\leq p<q\leq\infty.$ For every $x\in\ell^{q}$ 

(a) if $q\geq1$, then ${\mathrm{Orb}}\left(x;\ell^{p},\ell^{q}\right)={\ {{
\mathcal{K}}-{\mathrm{Orb}}}}(x;\ell^{p},\ell^{q})$;

(a) if $q<1$, then ${\mathrm{Orb}}\left(x;\ell^{p},\ell^{q}\right)={\mathrm{
Orb}}\left(x;\ell^{p},\ell^{1}\right)={{\mathcal{K}}-{\mathrm{Orb}}}
(x;\ell^{p},\ell^{1})$.

{Moreover, the quasi-norms of the above spaces are equivalent with constants
independent of $x\in \ell ^{q}$.}
\end{corollary}

Now, we are able to prove the following additivity property for orbits of
elements with respect to the couples $\left(\ell^{p},\ell^{q}\right)$, $
0\leq p<q\leq\infty$ (cf. \cite{BO06}).

\begin{proposition}
\label{orbits} Let $0\leq s<p<q<r\leq\infty.$ Then, for each $
x\in\ell^{p}+\ell^{q}$ we have 
\begin{equation*}
{\mathrm{Orb}}\left(x;\ell^{p},\ell^{q}\right)={\mathrm{Orb}}
\left(x;\ell^{s},\ell^{q}\right)+{\mathrm{Orb}}\left(x;\ell^{p},\ell^{r}\right),
\end{equation*}
and the quasinorms of these spaces are equivalent with a constant independent of $x$.
\end{proposition}

\begin{proof}
Observe that, by interpolation, with a constant independent of $x$ it follows 
\begin{equation*}
{\mathrm{Orb}}\left(x;\ell^{p},\ell^{q}\right)\supset{\mathrm{Orb}}
\left(x;\ell^{s},\ell^{q}\right)+{\mathrm{Orb}}\left(x;\ell^{p},\ell^{r}\right). 
\end{equation*}
Therefore, it remains to prove the opposite embedding.

Assume first that $q\geq1$. Then, if $y\in{\mathrm{Orb}}\left(x;\ell^{p},
\ell^{q}\right)$, from Corollary \ref{Corr-1} it follows that $y\in{{
\mathcal{K}}-{\mathrm{Orb}}}\left(x;\ell^{p},\ell^{q}\right)$. Then,
according to Theorem \ref{p3}, we may write $y=T_{0}x+T_{1}x$, where $
T_{0}\in\mathfrak{L}(\ell^{0},\ell^{q})$ and $T_{1}\in\mathfrak{L}
(\ell^{p},\ell^{\infty})$. Hence, again by interpolation $
T_{0}:\,\ell^{s}\rightarrow\ell^{s}$, $T_{1}:\,\ell^{r}\rightarrow\ell^{r}.$
This yields $T_{0}x\in{\mathrm{Orb}}\left(x;\ell^{s},\ell^{q}\right)$ and $
T_{1}x\in{\mathrm{Orb}}\left(x;\ell^{p},\ell^{r}\right)$. Thus, with a constant independent of $x$
\begin{equation*}
{\mathrm{Orb}}\left(x;\ell^{p},\ell^{q}\right)\subset{\mathrm{Orb}}
\left(x;\ell^{s},\ell^{q}\right)+{\mathrm{Orb}}\left(x;\ell^{p},\ell^{r}
\right), 
\end{equation*}
and in this case the result follows.

Let now $q<1.$ Applying successively Corollary \ref{Corr-1} and the above
result for the couple $(\ell^{p},\ell^{1})$ and the numbers $s<p<1<\infty$,
we get 
\begin{eqnarray*}
{\mathrm{Orb}}\left(x;\ell^{p},\ell^{q}\right) & = & {\mathrm{Orb}}
\left(x;\ell^{p},\ell^{1}\right) \\
& = & {\mathrm{Orb}}\left(x;\ell^{s},\ell^{1}\right)+{\mathrm{Orb}}
\left(x;\ell^{p},\ell^{\infty}\right) \\
& \subset & {\mathrm{Orb}}\left(x;\ell^{s},\ell^{q}\right)+{\mathrm{Orb}}
\left(x;\ell^{p},\ell^{r}\right),
\end{eqnarray*}
and everything is done.
\end{proof}

In the proof of the next result we follow the reasoning used in the paper 
\cite{BO06}.

\begin{proposition}
\label{B-O} Let $\left(X_{0},X_{3}\right)$ and $\left(X_{1},X_{2}\right)$ be
two couples of quasi-Banach Abelian groups such that $X_{2}\in
Int\left(X_{0},X_{3}\right)$ and $X_{3}\in Int\left(X_{1},X_{2}\right)$.
Moreover, assume that 
\begin{equation}
\label{equa1}
X_{2}\cap X_{3}\subseteq(X_{0}\cap X_{3})+(X_{1}\cap
X_{2}),\;\;(X_{0}+X_{3})\cap(X_{1}+X_{2})\subseteq X_{2}+X_{3}, 
\end{equation}
and for every $x\in X_{2}+X_{3}$
\begin{equation}
\label{equa2}
{\mathrm{Orb}}\left(x;X_{2},X_{3}\right)={\mathrm{Orb}}\left(x;X_{0},X_{3}
\right)+{\mathrm{Orb}}\left(x;X_{1},X_{2}\right).
\end{equation}
Then, 
\begin{equation*}
Int\left(X_{2},X_{3}\right)=Int\left(X_{0},X_{3}\right)\cap
Int\left(X_{1},X_{2}\right). 
\end{equation*}
\end{proposition}

\begin{proof}
Suppose first $X\in Int\left(X_{2},X_{3}\right)$. If $T\in\mathfrak{L}
(X_{0},X_{3})$ it follows by interpolation that $T:\,X_{2}\rightarrow X_{2}$
and hence $T\in\mathfrak{L}(X_{2},X_{3})$. This implies that $
T:\,X\rightarrow X$, i.e., $X\in Int\left(X_{0},X_{3}\right).$ In the same
manner one can check that $X\in Int\left(X_{1},X_{2}\right)$.

Conversely, let $X\in Int\left(X_{0},X_{3}\right)\cap
Int\left(X_{1},X_{2}\right)$. Then, by \eqref{equa1}, 
\begin{equation*}
X_{2}\cap X_{3}\subseteq(X_{0}\cap X_{3})+(X_{1}\cap X_{2})\subseteq
X\subseteq(X_{0}+X_{3})\cap(X_{1}+X_{2})\subseteq X_{2}+X_{3}, 
\end{equation*}
that is, $X$ is an intermediate space between $X_{2}$ and $X_{3}$.

Moreover, for each $x\in X$ we have ${\mathrm{Orb}}\left(x;X_{0},X_{3}
\right)\subset X$ and ${\mathrm{Orb}}\left(x;X_{1},X_{2}\right)\subset X$.
Hence, applying assumption \eqref{equa2}, we conclude that ${\mathrm{Orb}}\left(x;X_{2},X_{3}\right)\subset X$. Thus, $X\in Int\left(X_{2},X_{3}\right)$, and the desired result
follows.
\end{proof}

Now, we are ready to prove that the full scale of $\ell^{p}$-spaces, $0\leq
p\leq\infty$, possesses the Arazy-Cwikel property.

\begin{proof}[Proof of Theorem \protect\ref{int spaces}]
It suffices to apply Propositions \ref{orbits} and \ref{B-O}.
\end{proof}

\vskip0.5cm

In conclusion of this section we deduce the following result, which
contains, in particular, a solution of the conjecture stated by Levitina,
Sukochev and Zanin in the paper \cite{LSZ17a} (its earlier version may be
found in the preprint \cite{LSZ17}).

As above, let $\sigma_{E}$ denote the Aoki-Rolewicz index of the
quasi-Banach sequence space $E$ (see Section \ref{Prel-Interpolation}).

\begin{theorem}
\label{int properties} Let $q>0$ and let $E$ be a quasi-Banach sequence
space. The following assertions are equivalent:

(i) $E\in Int(\ell^{0},\ell^{q})$;

(ii) $E\in Int\left(\ell^{p},\ell^{q}\right)$ for each $p\in(0,\sigma_{E})$;

(iii) $E\in Int\left(\ell^{0},\ell^{q}\right)$ and $E\in
Int\left(\ell^{p},\ell^{\infty}\right)$ for each $p\in(0,\sigma_{E})$.
\end{theorem}

\begin{proof}
Observe that from (i) it follows that $E$ is symmetric (see e.g. \cite[
Lemma~1.11]{CSZ}). Therefore, the implication $(i)\Longrightarrow(iii)$ is
an immediate consequence of Proposition \ref{L3} (see also its proof). In
turn, the equivalence $(ii)\Longleftrightarrow(iii)$ follows from Theorem 
\ref{int spaces}. Since the implication $(iii)\Longrightarrow(i)$ is
obvious, the proof is completed.
\end{proof}

By the latter result, we are able to determine the exact assumptions under
which the above-mentioned Levitina-Sukochev-Zanin conjecture is resolved in
affirmative. To justify this claim, we consider the following conditions,
assuming that $q>0$ and $E$ is a quasi-Banach sequence space such that $
E\subset\ell^{\infty}$:

(a) for any $x=(x_{n})_{n=1}^{\infty}\in E$ and $y=(y_{n})_{n=1}^{\infty}\in
\ell^{\infty}$ such that 
\begin{equation}
\sum_{n=m}^{\infty}(y_{n}^{*})^{q}\le\sum_{n=m}^{\infty}(x_{n}^{*})^{q},\;
\;m\in\mathbb{N},  \label{LSZ}
\end{equation}
we have $y\in E$ and $\|y\|_{E}\le C\|x\|_{E}$, where $C$ depends only on $E$
and $q$;

(b) there exists $p\in(0,q)$ such that $E\in Int(\ell^{p},\ell^{q})$.

In \cite{LSZ17a}, the authors were asking whether the conditions (a) and (b) are equivalent, and then, in \cite{CSZ}, the affirmative answer to this question has been given in the case when $q\ge1$ (see also \cite{CN17}).

To show a connection of the above question with Theorem \ref{int properties}, suppose that elements $x\in E$ and $y\in\ell^{q}$ satisfy the condition 
\begin{equation*}
{\mathcal{K}}(t,y;\ell^{0},\ell^{q})\leq{\mathcal{K}}(t,x;\ell^{0},
\ell^{q}),\;t>0. 
\end{equation*}
Then, by \eqref{impl2} and  \eqref{EQ13ddd}, we have 
\begin{equation*}
\sum_{n=m}^{\infty}(y_{n}^{\ast})^{q}\leq 2^q\sum_{n=[(m-1)/2]+1}^{\infty}(x_{n}^{\ast})^{q},\;\;m\in\mathbb{N}. 
\end{equation*}
One can easily check that
\begin{equation*}
\sum_{n=[(m-1)/2]+1}^{\infty}(x_{n}^{\ast})^{q}\le \sum_{n=m}^{\infty}(D_{2}x^{\ast})_{n}^{q},\;\;m\in\mathbb{N}. 
\end{equation*}
In consequence, 
\begin{equation*}
\sum_{n=m}^{\infty}(y_{n}^{\ast})^{q}\leq 2^q\sum_{n=m}^{\infty}(D_{2}x^{\ast})_{n}^{q},\;\;m\in\mathbb{N}. 
\end{equation*}
If the condition (a) holds, then clearly the space $E$ is symmetric and hence from the latter inequality it follows easily that 
\begin{equation*}
\Vert y\Vert_{E}\leq2C\Vert D_{2}x\Vert_{E}\leq4C\Vert x\Vert_{E}, 
\end{equation*}
where $C$ is the constant from (a). 
Therefore, $E$ is a uniform ${\mathcal{K}}$-monotone space with respect to
the couple $(\ell^{0},\ell^{q})$ and hence $E\in
Int\left(\ell^{0},\ell^{q}\right)$. Thus, (a) implies condition (i) from
Theorem \ref{int properties}, and so equivalence $(i)\Longleftrightarrow(ii)$
of this theorem shows that the implication $(a)\Longrightarrow(b)$ holds, in fact, for each $q>0$ (including also the non-Calder\'{o}n-Mityagin range $0<q<1$).

Regarding the converse direction, the situation in the cases $q\ge1$ and $0<q<1$ is different. If $q\ge1$, then $(\ell^{0},\ell^{q})$ is a uniform Calder\'{o}n-Mityagin couple (see Corollary \ref{Ar-Cw}). Hence,
from (b) it follows that $E$ is a uniformly $K$-monotone space with respect to $(\ell^{0},\ell^{q})$, and so we have (a). Thus, in this case all conditions (a), (b) and (i) are equivalent.

If $0<q<1$, in general, the implication $(b)\Longrightarrow(a)$ is not longer true. Indeed, since the couple $(\ell^{0},\ell^{q})$ has not the Calder\'{o}n-Mityagin property (see Theorem \ref{Th:No-CM-Property} in the next section), there is a non-${\mathcal{K
}}$-monotone interpolation space $E_{0}$ between $\ell^{0}$ and $\ell^{q}$.
The above discussion shows then that $E_{0}$ does not satisfy the  condition
(a). At the same time, since from Theorem \ref{int properties} it follows
that $E_{0}\in Int\left(\ell^{p},\ell^{q}\right)$ for each $p\in(0,r_{E_{0}})
$, the condition (b) for $E_{0}$ is fulfilled. This proves the last
claim. 

\vskip 1cm

\section{\label{Section-non-CM}A characterization of couples $
\left(\ell^{p},\ell^{q}\right)$ with the Calder\'{o}n-Mityagin property.}

\label{fails CM}

Here, we prove one of the main results of this paper, showing that $
\left(\ell^{p},\ell^{q}\right)$ fails to be a Calder\'{o}n-Mityagin couple
whenever $0\leq p<q<1$. In fact, we establish even a stronger result, which
implies, in particular, that for every sequence $g\in\ell^{q}\setminus
\ell^{p}$ there exists a sequence $f\in\ell^{q}$ such that $g\in
K-Orb\left(f,\ell^{p},\ell^{q}\right)\setminus
Orb\left(f,\ell^{p},\ell^{q}\right)$. First, we introduce the following
useful notion.

\begin{definition}
Let $\vec{X}=(X_{0},X_{1})$ be a quasi-Banach couple (or more generally, a
couple of quasi-Banach Abelian groups) and let $y\in X_{0}+X_{1}$. Then, we say that $y$ is a \textit{Calder\'{o}n-Mityagin element} (\textit{CM-element}, in brief) with respect to $\vec{X}$ provided if the inequality
\begin{equation}
{\mathcal{K}}(t,y;X_{0},X_{1})\leq{\mathcal{K}}(t,x;X_{0},X_{1}),\;\;t>0,
\label{K-func ineq}
\end{equation}
which holds for an element $x\in X_{0}+X_{1}$, implies the existence of  an operator $T\in\mathfrak{L} (X_{0},X_{1})$ with $Tx=y.$ The set of all CM-elements with respect to the
couple $\vec{X}$ we will denote by $CM(\vec{X})=CM(X_{0},X_{1})$.
\end{definition}

Clearly, $CM(\vec{X})=X_{0}+X_{1}$ if and only if the couple $\vec{X}$ has
the Calder\'{o}n-Mityagin property.

\vskip0.5cm



\vskip0.5cm

Let a quasi-Banach couple $\vec{X}$ be such that the sum $X_{0}+X_{1}$ is
continuously embedded into a Banach space $Z$. Show that then we have 
\begin{equation}
\label{CM-emb}
CM(\vec{X})\supset X_{0}\cap X_{1}.
\end{equation}
Indeed, let inequality \eqref{K-func ineq}
to be hold for some $y\in X_{0}\cap X_{1}$ and $x\in X_{0}+X_{1}$. Clearly,
we may assume that $x\neq0$. Then, according to Hahn-Banach Theorem, there
is a linear functional $x^{\ast}\in Z^{\ast}$, $\Vert
x^{\ast}\Vert_{Z^{\ast}}=1/\Vert x\Vert_{Z}$, with $x^{\ast}(x)=1$. Now, if
the operator $T$ is defined by $Tu:=x^{\ast}(u)y$, then $T$ is bounded from $
X_{0}+X_{1}$ into $X_{0}\cap X_{1}$ and $Tx=y.$ Moreover, for $i=0,1$ we have
\begin{equation*}
\Vert Tu\Vert_{X_{i}}\leq|x^{\ast}(u)|\Vert y\Vert_{X_{i}}\leq\Vert
x^{\ast}\Vert_{Z^{*}}\Vert u\Vert_{Z}\Vert y\Vert_{X_{i}}\leq \frac{C\Vert y\Vert_{X_{i}}}{\Vert
x\Vert_{Z}}\Vert u\Vert_{X_{i}}, 
\end{equation*}
where $C$ is the embedding constant of $X_{0}+X_{1}$ into $Z$, and \eqref{CM-emb} is proved.

In particular, since $\ell^{r}\subset\ell^{1}$ continuously if $r\in(0,1)$,
we have that $\ell^{p}\subseteq CM(\ell^{p},\ell^{q})$ for all $
0<p<q\leq\infty$. Moreover, this result can be extended also to the extreme case $p=0$. Indeed, assuming that elements $x\in\ell^{q}$ and $y\in\ell^{0}$, $y\neq0$, satisfy \eqref{K-func ineq} in the case $\vec{X}=(\ell^{0},\ell^{q})$, one
can readily check that the operator $Tu:=x^{\ast}(u)y$, where $
x^{\ast}\in(\ell^{1})^{\ast}=\ell^{\infty}$ such that $\Vert
x^{\ast}\Vert_{\ell^{\infty}}=1/\Vert x\Vert_{\ell^{1}}$ and $x^{\ast}(x)=1$,
is a bounded homomorphism on $\ell^{0}$ and $\Vert
T\Vert_{\ell^{0}\rightarrow\ell^{0}}\leq\Vert y\Vert_{\ell^{0}}$.

\vskip0.5cm

The following result shows that for the couple $\left(\ell^{p},
\ell^{q}\right)$, $0\leq p<q<1$, embedding \eqref{CM-emb} turns into an equality.

\begin{theorem}
\label{Th:No-CM-Property} Let $0\leq p<q<1.$ Then, $CM\left(\ell^{p},
\ell^{q}\right)=\ell^{p}.$ In particular, $\left(\ell^{p},\ell^{q}\right)$
is not a Calder\'{o}n-Mityagin couple.
\end{theorem}

Applying Theorems \ref{Ar-Cw} and \ref{Th:No-CM-Property}, we get the
following characterization of couples $\left(\ell^{p},\ell^{q}\right)$ with
the Calder\'{o}n-Mityagin property.

\begin{corollary}
\label{th2} Let $0\le p<q\le\infty$. Then, the following conditions are
equivalent:

(i) $\left(\ell^{p},\ell^{q}\right)$ is a Calder\'{o}n-Mityagin couple;

(ii) $\left(\ell^{p},\ell^{q}\right)$ is a uniform Calder\'{o}n-Mityagin
couple;

(iii) $q\ge1$.
\end{corollary}

As we will see a little bit later, Theorem $\ref{Th:No-CM-Property}$ is a
straightforward consequence of the following result.

\begin{theorem}
\label{Th-main-estimate} Assume $0\le p<q<r\leq\infty.$ Let $
g=\left(g_{n}\right)_{n=1}^{\infty}\in\ell^{q}\setminus\ell^{p}$ be a
nonnegative and nonincreasing sequence. Then, there exists a nonnegative,
nonincreasing sequence $f=\left(f_{n}\right)_{n=1}^{\infty}$ such that 
\begin{equation}
f\in\ell^{q}\text{ and }\left\Vert f\right\Vert _{\ell^{q}}=\left\Vert
g\right\Vert _{\ell^{q}},  \label{th_main_estimate_lp}
\end{equation}
\begin{equation}
{\mathcal{K}}\left(t,g;\ell^{0},\ell^{q}\right)\leq{\mathcal{K}}
\left(t,f;\ell^{0},\ell^{q}\right),\;\;t>0,  \label{th_main_estimate_mon}
\end{equation}
but 
\begin{equation}
\lim\inf_{t\rightarrow\infty}\frac{{\mathcal{K}}\left(t,f;\ell^{p},\ell^{r}
\right)}{{\mathcal{K}}\left(t,g;\ell^{p},\ell^{r}\right)}=0.
\label{th_main_k_estimate}
\end{equation}
\end{theorem}

Before proceeding with the proof of this theorem, we deduce some its consequences.

\begin{corollary}
\bigskip{}  \label{Th_main_est_corr2} Let $0\le p<q<r\leq\infty$, $q\geq1$.
For any $g\in\ell^{q}\setminus\ell^{p}$ there exists $f\in\ell^{q}$ such
that 
\begin{equation*}
g\in{\mathrm{Orb}}\left(f,\ell^{s},\ell^{q}\right)\;\;\mbox{for each}
\;\;s\in\lbrack0,q) 
\end{equation*}
and 
\begin{equation*}
g\notin{\mathrm{Orb}}\left(f,\ell^{p},\ell^{r}\right). 
\end{equation*}
\end{corollary}

\begin{proof}
Clearly, we can assume that $f$ is nonnegative and nonincreasing. Then, by Theorem \ref{Th-main-estimate}, there exists $f\in\ell^{q}$ satisfying
conditions \eqref{th_main_estimate_mon} and \eqref{th_main_k_estimate}.
Recall that, for every $0<s<q$, $\ell^{s}$ is an interpolation space between $\ell^{0}$ and $\ell^{q}$, which can be obtained by applying the classical real $\mathcal{K}$-method to the couple $(\ell^{0},\ell^{q})$ (see Section 
\ref{Prel-sequence-spaces} and \cite[Theorem~7.1.7]{BL76}). Therefore,
combining inequality \eqref{th_main_estimate_mon} together with the
well-known reiteration theorem (see e.g. \cite[Theorem~3.1.2, p.~332]{BK91}), we can find, for every $s\in\lbrack0,q)$, a constant $C_{q,s}$ such that 
\begin{equation}
{\mathcal{K}}\left(t,g;\ell^{s},\ell^{q}\right)\leq C_{q,s}\cdot{\mathcal{K}}
\left(t,f;\ell^{s},\ell^{q}\right),\;\;t>0.  \label{K-func ineq1}
\end{equation}
Therefore, since $q\geq1$, from Theorem \ref{Ar-Cw} it follows that $g\in{
\mathrm{Orb}}\left(f,\ell^{s},\ell^{q}\right).$

On the other hand, let us assume that $g\in{\mathrm{Orb}}\left(f;\ell^{p},
\ell^{r}\right)$, that is, $g=Wf$ for some operator $W\in\mathfrak{L}
(\ell^{p},\ell^{r})$. Hence, 
\begin{equation*}
{\mathcal{K}}\left(t,g;\ell^{p},\ell^{r}\right)={\mathcal{K}}
\left(t,Wf;\ell^{p},\ell^{r}\right)\leq\left\Vert W\right\Vert _{\mathfrak{L}
(\ell^{p},\ell^{r})}{\mathcal{K}}\left(t,f;\ell^{p},\ell^{r}\right),\;\;t>0, 
\end{equation*}
which contradicts \eqref{th_main_k_estimate}.
\end{proof}

\begin{proof}[Proof of Theorem \protect\ref{Th:No-CM-Property}]
Let $0\leq p<q<1$. We need to show that every $g\in\ell^{q}\setminus\ell^{p}$ is not a CM-element with respect to the couple $(\ell^{p},\ell^{q})$.

Given $g\in\ell^{q}\setminus\ell^{p}$ we select $f\in\ell^{q}$ as in Theorem 
\ref{Th-main-estimate} with $r=1$. Then, reasoning in the same way as in the
proof of the preceding corollary, we conclude that for each $s\in[0,q)$
inequality \eqref{K-func ineq1} holds with some constant $C_{q,s}$.

On the other hand, we claim that 
\begin{equation}
g\notin{\mathrm{Orb}}\left(f;\ell^{p},\ell^{r^{\prime}}\right)\;\;\mbox{for
every}\;\;r^{\prime}\in\lbrack q,1).  \label{K-func ineq1extra}
\end{equation}
Indeed, to get the contrary, assume that $g\in{\mathrm{Orb}}\left(f,\ell^{p},
\ell^{r^{\prime}}\right)$ for some $r^{\prime}\in\lbrack q,1)$. This means
that there exists an operator $V\in\mathfrak{L}(\ell^{p},\ell^{r^{\prime}})$
such that $Vf=g.$ Since $r^{\prime}<1$, by Theorem \ref{Th-bounded-lp}, $V$
has a linear extension $\widetilde{V}:\,\ell^{1}\rightarrow\ell^{1}$ such
that $\Vert\widetilde{V}\Vert_{\ell^{1}\rightarrow\ell^{1}}\leq\Vert{V}
\Vert_{\ell^{r^{\prime}}\rightarrow\ell^{r^{\prime}}}$. Then, 
\begin{equation*}
{\mathcal{K}}\left(t,g;\ell^{p},\ell^{1}\right)={\mathcal{K}}(t,\widetilde{V}
f;\ell^{p},\ell^{1})\leq\left\Vert V\right\Vert _{\mathfrak{L}
(\ell^{p},\ell^{r^{\prime}})}{\mathcal{K}}\left(t,f;\ell^{p},\ell^{1}
\right),\;\;t>0, 
\end{equation*}
which is a contradiction in view of \eqref{th_main_k_estimate} (with $r=1$).

Now, setting $s=p$ and $r^{\prime}=q$ in \eqref{K-func ineq1} and 
\eqref{K-func ineq1extra} respectively, we see that every element $
g\in\ell^{q}\setminus\ell^{p}$ fails to be a CM-element with respect to the
couple $\left(\ell^{p},\ell^{q}\right)$. Therefore, by \eqref{CM-emb}, we have $CM\left(\ell^{p},\ell^{q}\right)=\ell^{p}.$

The second assertion of the theorem that $\left(\ell^{p},\ell^{q}\right)$ is
not a Calder\'{o}n-Mityagin couple is now almost obvious. Indeed, take for $g$
any element from $\ell^{q}\setminus\ell^{p}$ (say, $g=(n^{-\sigma})_{n=1}^{
\infty}$, where $\sigma\in(1/q,1/p)$). Since $g\not\in
CM\left(\ell^{p},\ell^{q}\right)$, there is $f\in\ell^{q}$ such that $g\in{
\mathcal{K}}-{\mathrm{Orb}}\left(f;\ell^{p},\ell^{q}\right)\setminus{\mathrm{
Orb}}\left(f;\ell^{p},\ell^{q}\right)$. Clearly, this yields the desired
result.
\end{proof}

\vskip 0.5cm

Now, we proceed with the proof of Theorem \ref{Th-main-estimate}.

Let $a,b$ be arbitrary positive integers. Let $T_{a,b}$ be the linear map
defined on every function $h:\,[0,\infty)\to\mathbb{R}$ as follows: 
\begin{equation*}
T_{a,b}\left(h\right)(x):=\left\{ 
\begin{array}{c}
h\left(x\right)\;\;\mbox{if}\;\;0<x<a \\ 
b^{-1/q}h\left(a+\frac{x-a}{b}\right)\;\;\mbox{if}\;\;a\leq x<\infty.
\end{array}
\right. 
\end{equation*}
Denote by $\Phi$ the collection of all functions $h:\,[0,\infty)\to\mathbb{R}
$ which are constant on each interval $\left[n-1,n\right)$, $n\in\mathbb{N}.$
Then $T_{a,b}:\Phi\rightarrow\Phi$. Moreover, if $h$ is a nonnegative and
nonincreasing function, then $T_{a,b}h$ is also nonnegative and
nonincreasing.

Further properties of operators $T_{a,b}$ we will need are collected in
the following lemma.

\begin{lemma}
\label{Lemma_tab} Assume that $0<p<q<r<\infty$ and $h\in\Phi$ is a
nonnegative, nonincreasing function. Then for any $a,b\in\mathbb{N}$ 
\begin{equation}
\lim_{b\rightarrow\infty}\int\nolimits
_{0}^{t}\left(T_{a,b}h\left(x\right)\right)^{p}\,dx=\int\nolimits
_{0}^{a}h\left(x\right)^{p}\,dx,\;\;t>a,  \label{lemma_tab_1}
\end{equation}
\begin{equation}
\int\nolimits
_{t}^{\infty}\left(T_{a,b}h\left(x\right)\right)^{q}\,dx\geq\int\nolimits
_{t}^{\infty}h\left(x\right)^{q}\,dx,\;\;t>0,  \label{lemma_tab_2}
\end{equation}
\begin{equation}
\int\nolimits
_{t}^{\infty}\left(T_{a,b}h\left(x\right)\right)^{q}\,dx=\int\nolimits
_{t}^{\infty}h\left(x\right)^{q}\,dx,\;\;0\leq t\leq a.  \label{lemma_tab_3}
\end{equation}

If additionally $h\in L^{r}\left(0,\infty\right)$, then $T_{a,b}h\in
L^{r}\left(0,\infty\right)$, and 
\begin{equation}
\lim_{b\rightarrow\infty}\int\nolimits
_{a}^{\infty}\left(T_{a,b}h\left(x\right)\right)^{r}\,dx=0.
\label{lemma_tab_4}
\end{equation}

If $\int\nolimits _{t}^{\infty}h\left(x\right)^{r}\,dx>0$ for all $t>0$,
then we have 
\begin{equation}
\int\nolimits
_{t}^{\infty}\left(T_{a,b}h\left(x\right)\right)^{r}\,dx\leq\int\nolimits
_{t}^{\infty}h\left(x\right)^{r}\,dx  \label{lemma_tab_5}
\end{equation}
whenever a positive integer $b$ satisfies the condition 
\begin{equation}
b>\left(\frac{\int\nolimits _{a}^{\infty}h\left(y\right)^{r}dy}{
\int\nolimits _{t}^{\infty}h\left(y\right)^{r}dy}\right)^{\frac{q}{r-q}}.
\label{lemma_tab_5a}
\end{equation}
\end{lemma}

We postpone the proof of this lemma, which is a series of elementary
calculations, until the end of this section.

\begin{proof}[Proof of Theorem \protect\ref{Th-main-estimate}]
As above, for some technical purposes, it will be convenient to consider,
instead of sequences $h=\left(h_{n}\right)_{n=1}^{\infty}$, the step
functions $\widetilde{h}(t)$ defined for $t>0$ by 
\begin{equation*}
\widetilde{h}(t):=\sum_{n=1}^{\infty}h_{n}\chi_{\lbrack n-1,n)}(t). 
\end{equation*}

We will use further an infinite composition of suitable operators of the
form $T_{a,b}$ applied to the function $\widetilde{g}=\sum_{n=1}^{
\infty}g_{n}\chi_{[n-1,n)}$ to obtain as a result the nonnegative and
nonincreasing function $\widetilde{f}\in\Phi$ such that the corresponding
sequence $f=\left(f_{n}\right)_{n=1}^{\infty}$ will possess the required
properties \eqref{th_main_estimate_lp}, \eqref{th_main_estimate_mon} and 
\eqref{th_main_k_estimate}.

Let $\left(a_{n}\right)_{n=1}^{\infty}$ and $\left(b_{n}\right)_{n=1}^{
\infty}$ be two sequences of positive integers such that the sequence $
\left(a_{n}\right)_{n=1}^{\infty}$ is strictly increasing and hence $
\lim_{n\rightarrow\infty}a_{n}=\infty.$ Next, choosing these sequences in a
special way, we consider the sequence of functions $\left(G_{n}
\right)_{n=1}^{\infty}$ defined by $G_{1}=T_{a_{1},b_{1}}(\widetilde{g})$, $G_{n}=T_{a_{n},b_{n}}\left(G_{n-1}\right)$ for $n\geq2$. Then, by the
definition of $T_{a,b}$, 
\begin{equation}
G_{n}\left(x\right)=G_{m}\left(x\right),\;\;\mbox{for each}\;\;n\geq m\;
\mbox{ and for all}\;x\in\left[0,a_{m+1}\right).  \label{extra1}
\end{equation}
Since $\bigcup_{m=1}^{\infty}\left[0,a_{m+1}\right)=[0,\infty)$, there
exists the pointwise limit 
\begin{equation*}
\widetilde{f}\left(x\right):=\lim_{n\rightarrow\infty}G_{n}\left(x\right),\;
\;x>0. 
\end{equation*}
From \eqref{extra1} it follows that 
\begin{equation}
\widetilde{f}\left(x\right)=G_{n}\left(x\right),\;\;\mbox{for each}
\;\;n\ge m\;\;\mbox{if}\;\;x\in\left[0,a_{m+1}\right).  \label{extra2}
\end{equation}
It is clear also that $\widetilde{f}\in\Phi$ and it is a nonnegative and
nonincreasing function. Now, the main our task is to show that, if the
auxiliary sequences $\left(a_{n}\right)_{n=1}^{\infty}$ and $
\left(b_{n}\right)_{n=1}^{\infty}$ are constructed suitably, the sequence $
\left(f_{n}\right)_{n=1}^{\infty}$ of values of $\widetilde{f}$ will satisfy
the conditions \eqref{th_main_estimate_lp}, \eqref{th_main_estimate_mon} and 
\eqref{th_main_k_estimate}.

At this stage we will assume that $0<p<r<\infty$. Let us establish first
some properties of the functions $G_{n}$ and $\widetilde{f}$.

Since $g\in\ell^{q}$, from Lemma \ref{Lemma_tab} (see \eqref{lemma_tab_3})
it follows that $G_{n}\in L^{q}\left(0,\infty\right)$ and 
\begin{equation}
\int\nolimits _{0}^{\infty}G_{n}\left(x\right)^{q}\,dx=\int\nolimits
_{0}^{\infty}\widetilde{g}\left(x\right)^{q}\,dx=\sum_{m=1}^{
\infty}g_{m}^{q},\;\;n\in\mathbb{N}.  \label{extra3}
\end{equation}
Consequently, for every fixed positive integer $n$ we have 
\begin{equation*}
\lim_{y\rightarrow\infty}\int\nolimits _{y}^{\infty}G_{n}(x)^{q}\,dx=0, 
\end{equation*}
and hence the set 
\begin{equation*}
\left\{ m\in\mathbb{N}:\,\int\nolimits
_{m}^{\infty}G_{n}\left(x\right)^{q}\,dx\leq\frac{1}{n}\right\} 
\end{equation*}
is non-empty. Let 
\begin{equation}
\gamma_{n}:=\min\left\{ m\in\mathbb{N}:\,\int\nolimits
_{m}^{\infty}G_{n}\left(x\right)^{q}\,dx\leq\frac{1}{n}\right\} .
\label{extra6}
\end{equation}

Next, by \eqref{extra2}, for each $t\ge0$ and all $n$ satisfying $
a_{n+1}\geq t$, we have 
\begin{eqnarray}
\int\nolimits _{t}^{\infty}G_{n}\left(x\right)^{q}\,dx & = & \int\nolimits
_{t}^{a_{n+1}}G_{n}\left(x\right)^{q}\,dx+\int\nolimits
_{a_{n+1}}^{\infty}G_{n}\left(x\right)^{q}\,dx  \label{Th_est_1} \\
& = & \int\nolimits _{t}^{a_{n+1}}\widetilde{f}\left(x\right)^{q}\,dx+\int
\nolimits _{a_{n+1}}^{\infty}G_{n}\left(x\right)^{q}\,dx.  \notag
\end{eqnarray}
Observe that for each integer $n\geq2$ the function $G_{n}$ depends only on
the given function $\widetilde{g}$ and previously chosen $a_{k},b_{k}$, with 
$1\leq k\leq n.$ Consequently, after completing the first $n$ steps, we may
select $a_{n+1}$ so that 
\begin{equation*}
a_{n+1}\geq\gamma_{n}, 
\end{equation*}
where $\gamma_{n}$ is the positive integer defined in \eqref{extra6}. Then, 
\begin{equation*}
\int\nolimits _{a_{n+1}}^{\infty}G_{n}\left(x\right)^{q}\,dx\leq\frac{1}{n}, 
\end{equation*}
and hence, passing to the limit as $n$ tends to infinity in \eqref{Th_est_1}
, we get 
\begin{equation}
\int\nolimits _{t}^{\infty}\widetilde{f}\left(x\right)^{q}\,dx=\lim_{n
\rightarrow\infty}\int\nolimits _{t}^{\infty}G_{n}\left(x\right)^{q}\,dx\;\;
\mbox{for each}\;t\ge0.  \label{Th_est_2}
\end{equation}
In particular, setting $t=0$, by \eqref{extra3}, we have 
\begin{equation*}
\sum_{m=1}^{\infty}f_{m}^{q}=\int\nolimits _{0}^{\infty}\widetilde{f}
\left(x\right)^{q}\,dx=\int\nolimits _{0}^{\infty}\widetilde{g}
\left(x\right)^{q}\,dx=\sum_{m=1}^{\infty}g_{m}^{q}<\infty. 
\end{equation*}
Thus, the first desired condition, \eqref{th_main_estimate_lp}, is
established.

Furthermore, repeated applications of \eqref{lemma_tab_2} imply that 
\begin{equation*}
\int\nolimits _{t}^{\infty}G_{n}\left(x\right)^{q}\,dx\geq\int\nolimits
_{t}^{\infty}\widetilde{g}\left(x\right)^{q}\,dx\;\;\mbox{for every}\;\;n\in
\mathbb{N}\;\;\mbox{and all}\;\;t>0. 
\end{equation*}
Hence, taking limits and using \eqref{Th_est_2} yields 
\begin{equation*}
\int\nolimits _{t}^{\infty}\widetilde{f}\left(x\right)^{q}\,dx\geq\int
\nolimits _{t}^{\infty}\widetilde{g}\left(x\right)^{q}\,dx,\;\;t>0, 
\end{equation*}
which implies that 
\begin{equation*}
\sum_{j=m}^{\infty}g_{j}^{q}\leq\sum_{j=m}^{\infty}f_{j}^{q}\;\;\mbox{for
all}\;m\in\mathbb{N}. 
\end{equation*}
Now, recalling that for each nonnegative, nonincreasing sequence $
h=\left(h_{n}\right)_{n=1}^{\infty}$ we have 
\begin{equation*}
{\mathcal{K}}(t,h;\ell^{0},\ell^{q})=\inf_{m=0,1,2,\dots}\left(m+t\cdot
\mathcal{E}(m,h;\ell^{0},\ell^{q})\right),\;\;t>0, 
\end{equation*}
where 
\begin{equation*}
\mathcal{E}(m,h;\ell^{0},\ell^{q})=\left(\sum_{j=m+1}^{\infty}h_{j}^{q}
\right)^{1/q} 
\end{equation*}
(see e.g. \cite[§\thinspace 7.1]{BL76} or formulas \eqref{EQ13d} and 
\eqref{EQ13ddd} in Section \ref{Holmstedt}), we conclude that the second
required condition \eqref{th_main_estimate_mon} holds as well.

Thus, it remains only to prove (imposing suitable additional hypotheses on
the sequences $\left(a_{n}\right)_{n=1}^{\infty}$ and $\left(b_{n}
\right)_{n=1}^{\infty}$) the last required condition 
\eqref{th_main_k_estimate}. Here, we will use the assumption that $
g\notin\ell^{p}$, and so 
\begin{equation}
\widetilde{g}\notin L^{p}\left(0,\infty\right).  \label{extra4}
\end{equation}
Along with the previously introduced sequence $\left(\gamma_{n}
\right)_{n=1}^{\infty}$ we will need a new sequence $\left(\delta_{n}
\right)_{n=1}^{\infty}$. For definiteness, we set $a_{1}=b_{1}=\delta_{1}=1$
and hence $G_{1}=T_{1,1}\widetilde{g}=\widetilde{g}.$ Moreover, as was
specified above, $\gamma_{1}$ is the least positive integer $m$ satisfying
the inequality $\int_{m}^{\infty}G_{1}(x)^{q}\,dx\le1$.

Suppose that $n\geq1$ and $a_{k},b_{k},\delta_{k},G_{k},\gamma_{k}$ are
determined for all $1\leq k\leq n.$ To pass to the next step, we first set 
\begin{equation}
a_{n+1}:=\max\left(\gamma_{n},\delta_{n},a_{n}+1\right).  \label{th_def_new}
\end{equation}
Let now $\delta_{n+1}$ be the smallest positive integer with the following
properties: 
\begin{equation}
\delta_{n+1}>na_{n+1}  \label{extra imp}
\end{equation}
and 
\begin{equation*}
\int\nolimits _{0}^{\delta_{n+1}}\widetilde{g}\left(x\right)^{p}\,dx
\geq2n^{p}\int\nolimits _{0}^{a_{n+1}}G_{n}\left(x\right)^{p}\,dx. 
\end{equation*}
Observe that such an integer exists because of \eqref{extra4}. Therefore, by 
\eqref{lemma_tab_1}, we infer 
\begin{equation}
\lim_{b\rightarrow\infty}\int\nolimits
_{0}^{\delta_{n+1}}\left(T_{a_{n+1},b}G_{n}\left(x\right)\right)^{p}\,dx=
\int\nolimits _{0}^{a_{n+1}}G_{n}\left(x\right)^{p}\,dx\leq\frac{1}{2n^{p}}
\int\nolimits _{0}^{\delta_{n+1}}\widetilde{g}\left(x\right)^{p}\,dx.
\label{Th_est_3}
\end{equation}

Combining the fact that $\widetilde{g}$ is nonincreasing with \eqref{extra4}, we see that $\widetilde{g}\left(x\right)>0$ and hence 
\begin{equation*}
\int\nolimits _{0}^{\delta_{n+1}}\widetilde{g}\left(x\right)^{p}\,dx>0. 
\end{equation*}
Moreover, since the functions $\widetilde{g}$ and $G_{n}$ belong to the
intersection $\Phi\cap L^{q}\left(0,\infty\right)$, the embedding $
\ell^{q}\subset\ell^{r}$ ensures that 
\begin{equation}
\widetilde{g}\in L^{r}\left(0,\infty\right)\;\;\mbox{and}\;\;G_{n}\in
L^{r}\left(0,\infty\right).  \label{extra5}
\end{equation}
Then, using \eqref{lemma_tab_4} with the function $G_{n}$ and the number $
a_{n+1}$ instead of $h$ and $a$ respectively, we conclude 
\begin{equation}
\lim_{b\rightarrow\infty}\int\nolimits
_{a_{n+1}}^{\infty}\left(T_{a_{n+1},b}G_{n}\left(x\right)\right)^{r}\,dx=0
\label{Th_est_5}
\end{equation}

As was noted, the function $\widetilde{g}$ is strictly positive at all
points of $(0,\infty)$. Hence, $G_{n}\left(x\right)>0$ if $x>0.$ Therefore,
in view of \eqref{extra5}, we can apply the last result of Lemma \ref
{Lemma_tab}, according to that, for each $t>0$, every positive integer $b$,
satisfying the condition 
\begin{equation*}
b>\left(\frac{\int\nolimits _{a_{n+1}}^{\infty}G_{n}\left(y\right)^{r}\,dy}{
\int\nolimits _{t}^{\infty}G_{n}\left(y\right)^{r}\,dy}\right)^{q/\left(r-q
\right)}, 
\end{equation*}
satisfies also the inequality 
\begin{equation*}
\int\nolimits
_{t}^{\infty}\left(T_{a_{n+1},b}G_{n}\left(x\right)\right)^{r}\,dx\leq\int
\nolimits _{t}^{\infty}G_{n}\left(x\right)^{r}\,dx. 
\end{equation*}
Thus, from the condition 
\begin{equation*}
b>\left(\frac{\int\nolimits _{a_{n+1}}^{\infty}G_{n}\left(y\right)^{r}dy}{
\int\nolimits _{a_{n}}^{\infty}G_{n}\left(y\right)^{r}dy}\right)^{q/
\left(r-q\right)} 
\end{equation*}
and the fact that $a_{k}<a_{n}$ for $k<n$ it follows that 
\begin{equation}
\int\nolimits
_{a_{k}}^{\infty}\left(T_{a_{n+1},b}G_{n}\left(x\right)\right)^{r}\,dx\leq
\int\nolimits _{a_{k}}^{\infty}G_{n}\left(x\right)^{r}\,dx\;\;\mbox{for all}
\;\;1\leq k\leq n.  \label{Th_est_4}
\end{equation}
Summarizing all, we see that, in view of \eqref{Th_est_4},  \eqref{Th_est_3} and \eqref{Th_est_5}, it is possible to select an integer $b_{n+1}$ such
that the function $G_{n+1}=T_{a_{n+1},b_{n+1}}(G_{n})$ satisfies the following conditions: 
\begin{eqnarray}
\int\nolimits _{a_{k}}^{\infty}G_{n+1}\left(x\right)^{r}\,dx & \leq &
\int\nolimits _{a_{k}}^{\infty}G_{n}\left(x\right)^{r}\,dx\;\;\mbox{for all}
\;1\leq k\leq n,  \label{Th_est_6} \\
\int\nolimits _{0}^{\delta_{n+1}}G_{n+1}\left(x\right)^{p}\,dx & \leq & 
\frac{1}{n^{p}}\int\nolimits _{0}^{\delta_{n+1}}\widetilde{g}
\left(x\right)^{p}\,dx,  \label{Th_est_7} \\
\int\nolimits _{a_{n+1}}^{\infty}G_{n+1}\left(x\right)^{r}\,dx & \leq & 
\frac{1}{n^{r}}\int\nolimits _{\delta_{n+1}}^{\infty}\widetilde{g}
\left(x\right)^{r}\,dx.  \label{Th_est_8}
\end{eqnarray}

Having chosen $a_{n+1},\delta_{n+1},b_{n+1}$ and so $G_{n+1}$, we determine
then the number $\gamma_{n+1}$ by formula \eqref{extra6}. This means that
the step $n+1$ is completed. Proceeding in the same way, we obtain the
sequences $(a_{n})_{n=1}^{\infty}$, $(\delta_{n})_{n=1}^{\infty}$, $
(b_{n})_{n=1}^{\infty}$, $(G_{n})_{n=1}^{\infty}$ and $(\gamma_{n})_{n=1}^{
\infty}$ (in view of \eqref{extra4}, this process does not finish after only
finitely many steps). Moreover, as was observed above, there exists the
pointwise limit $\widetilde{f}$ of the sequence $\{G_{n}\}_{n=1}^{\infty}$
on $(0,\infty)$. Note that, thanks to the choice of $a_{n+1}$ determined by 
\eqref{th_def_new}, the reasoning from the first part of the proof shows
that the sequence $(f_{n})_{n=1}^{\infty}$ of values of the function $
\widetilde{f}$ satisfies conditions \eqref{th_main_estimate_lp} and 
\eqref{th_main_estimate_mon}. Therefore, it remains only to deduce 
\eqref{th_main_k_estimate}.

By \eqref{Th_est_7}, \eqref{extra2} and the inequality $\delta_{n+1}\leq
a_{n+2}$ (see \eqref{th_def_new}), it follows that 
\begin{equation}
\int\nolimits _{0}^{\delta_{n+1}}\widetilde{f}\left(x\right)^{p}\,dx\leq
\frac{1}{n^{p}}\int\nolimits _{0}^{\delta_{n+1}}\widetilde{g}
\left(x\right)^{p}\,dx.  \label{Th_est_9}
\end{equation}
Moreover, applying \eqref{Th_est_6}, for each integer $n$ and $m\geq n+1$ we
have 
\begin{equation*}
\int\nolimits
_{a_{n+1}}^{\infty}G_{m+1}\left(x\right)^{r}\,dx\leq\int\nolimits
_{a_{n+1}}^{\infty}G_{m}\left(x\right)^{r}\,dx, 
\end{equation*}
and hence by iteration 
\begin{equation*}
\int\nolimits
_{a_{n+1}}^{\infty}G_{m}\left(x\right)^{r}\,dx\leq\int\nolimits
_{a_{n+1}}^{\infty}G_{n+1}\left(x\right)^{r}\,dx\;\;\mbox{for each}\;m\geq
n+1. 
\end{equation*}
Consequently, for each $R\geq a_{n+1}$ and all $m\geq n+1$, in view of $
\eqref{Th_est_8}$, we infer that 
\begin{equation*}
\int\nolimits _{a_{n+1}}^{R}G_{m}\left(x\right)^{r}\,dx\leq\int\nolimits
_{a_{n+1}}^{\infty}G_{n+1}\left(x\right)^{r}\,dx\leq\frac{1}{n^{r}}
\int\nolimits _{\delta_{n+1}}^{\infty}\widetilde{g}\left(x\right)^{r}\,dx. 
\end{equation*}
Since $0\leq G_{m}\left(x\right)\leq\widetilde{g}\left(0\right)$ for all $
x\geq0$ and $m\in\mathbb{N}$, using dominated convergence on the interval $
\left[a_{n+1},R\right]$, we have 
\begin{equation*}
\int\nolimits _{a_{n+1}}^{R}\widetilde{f}\left(x\right)^{r}\,dx=\lim_{m
\rightarrow\infty}\int\nolimits _{a_{n+1}}^{R}G_{m}\left(x\right)^{r}\,dx\leq
\frac{1}{n^{r}}\int\nolimits _{\delta_{n+1}}^{\infty}\widetilde{g}
\left(x\right)^{r}\,dx, 
\end{equation*}
and then, passing to the limit as $R\rightarrow\infty$, we obtain 
\begin{equation}
\int\nolimits _{a_{n+1}}^{\infty}\widetilde{f}\left(x\right)^{r}\,dx\leq
\frac{1}{n^{r}}\int\nolimits _{\delta_{n+1}}^{\infty}\widetilde{g}
\left(x\right)^{r}\,dx.  \label{Th_est_10}
\end{equation}

Let $1/\alpha=1/p-1/r$. By the Holmstedt formula \eqref{Holmstedt1}, from \eqref{extra imp}, 
\eqref{Th_est_9} and \eqref{Th_est_10} it follows that\textbf{\ } 
\begin{eqnarray*}
& & {\mathcal{K}}(\delta_{n+1}^{1/\alpha},\widetilde{f};L^{p}(0,
\infty),L^{r}(0,\infty)) \\
& \leq & \left(\int\nolimits _{0}^{\delta_{n+1}}\widetilde{f}
\left(x\right)^{p}\,dx\right)^{1/p}+\delta_{n+1}^{1/\alpha}\left(\int
\nolimits _{\delta_{n+1}}^{\infty}\widetilde{f}\left(x\right)^{r}\,dx
\right)^{1/r} \\
& \leq & \left(\frac{1}{n^{p}}\int\nolimits _{0}^{\delta_{n+1}}\widetilde{g}
\left(x\right)^{p}\,dx\right)^{1/p}+\delta_{n+1}^{1/\alpha}\left(\frac{1}{
n^{r}}\int\nolimits _{\delta_{n+1}}^{\infty}\widetilde{g}\left(x\right)^{r}
\,dx\right)^{1/r} \\
& = & \frac{1}{n}\left(\left(\int\nolimits _{0}^{\delta_{n+1}}\widetilde{g}
\left(x\right)^{p}\,dx\right)^{1/p}+\delta_{n+1}^{1/\alpha}\left(\int
\nolimits _{\delta_{n+1}}^{\infty}\widetilde{g}\left(x\right)^{r}\,dx
\right)^{1/r}\right) \\
& \leq & \frac{C_{p,r}}{n}{\mathcal{K}}(\delta_{n+1}^{1/\alpha},\widetilde{g}
;L^{p}(0,\infty),L^{r}(0,\infty)),\;n=1,2,\dots
\end{eqnarray*}
Hence, for the corresponding sequences $f$ and $g$ it holds 
\begin{equation*}
\lim_{n\rightarrow\infty}\frac{{\mathcal{K}}\left(\delta_{n+1}^{1/\alpha},f;
\ell^{p},\ell^{r}\right)}{{\mathcal{K}}\left(\delta_{n+1}^{1/\alpha},g;
\ell^{p},\ell^{r}\right)}=0. 
\end{equation*}
Since $\delta_{n+1}>a_{n+1}$, then \eqref{th_def_new} ensures that $
\delta_{n+1}\rightarrow0$ as $n\rightarrow\infty$. Thus, relation 
\eqref{th_main_k_estimate} is established, and thus the proof of the theorem
in the case when $0<p<r<\infty$ is completed.

Consider now the case $0<p<r=\infty$. The only change which should be made
in the above construction of the sequences $\left(a_{n}\right)_{n=1}^{\infty}
$, $\left(\delta_{n}\right)_{n=1}^{\infty}$, $\left(b_{n}\right)_{n=1}^{
\infty}$, $\left(G_{n}\right)_{n=1}^{\infty}$ and $\left(\gamma_{n}
\right)_{n=1}^{\infty}$ is that in the step $n+1$ the value of $b_{n+1}$ is
required to satisfy only one condition, namely \eqref{Th_est_7}. Clearly,
then estimate \eqref{Th_est_9} follows in the same way. Combining it with
the Holmstedt formula \eqref{Holmstedt1a}, 
we get, as above, that 
\begin{equation*}
\lim_{n\rightarrow\infty}\frac{{\mathcal{K}}\left(\delta_{n+1}^{1/p},f;
\ell^{p},\ell^{\infty}\right)}{{\mathcal{K}}\left(\delta_{n+1}^{1/p},g;
\ell^{p},\ell^{\infty}\right)}=0. 
\end{equation*}
Thus, \eqref{th_main_k_estimate} is proved. Since the proofs of 
\eqref{th_main_estimate_lp} and \eqref{th_main_estimate_mon} do not require
any modifications in this case, the result follows.

Let now $0=p<r<\infty$. In this case in the step $n+1$ we choose the
value of $b_{n+1}$ so that only conditions \eqref{Th_est_6} and 
\eqref{Th_est_8} to be satisfied. Then, we obtain \eqref{Th_est_10}, which
being combined with inequality \eqref{extra imp}, gives us 
\begin{equation*}
\int\nolimits _{a_{n+1}}^{\infty}\widetilde{f}\left(x\right)^{r}\,dx\leq
\frac{1}{n^{r}}\int\nolimits _{na_{n+1}}^{\infty}\widetilde{g}
\left(x\right)^{r}\,dx. 
\end{equation*}
Consequently, we have 
\begin{equation}
\sum_{k=a_{n+1}+1}^{\infty}f_{k}^{r}\leq\frac{1}{n^{r}}\sum_{k=na_{n+1}+1}^{
\infty}g_{k}^{r},\;\;n\in\mathbb{N}.  \label{extra7}
\end{equation}

Let us assume that \eqref{th_main_k_estimate} does not hold, i.e., there is a constant $C>0$ such that 
\begin{equation}
{\mathcal{K}}\left(t,g;\ell^{0},\ell^{r}\right)\leq C{\mathcal{K}}
\left(t,f;\ell^{0},\ell^{r}\right)\;\;\mbox{for all}\;\;t>0.  \label{extra9}
\end{equation}
Then, by implication \eqref{impl2} and formula \eqref{EQ13ddd}, 
we have 
\begin{equation*}
\sum_{k=m}^{\infty}g_{k}^{r}\leq (2C)^r\sum_{k=[(m-1)/(2C)]+1}^{\infty}f_{k}^{r},\;\;m\in
\mathbb{N}. 
\end{equation*}
Substituting $m=([2C]+1)a_{n+1}+1$, $n\in\mathbb{N}$, in the
latter inequality, we get 
\begin{equation*}
\sum_{k=([2C]+1)a_{n+1}+1}^{\infty}g_{k}^{r}\leq
(2C)^r\sum_{k=a_{n+1}+1}^{\infty}f_{k}^{r},\;\;n\in\mathbb{N}. 
\end{equation*}
Combining this together with \eqref{extra7}, we come to the estimate 
\begin{equation*}
\sum_{k=([2C]+1)a_{n+1}+1}^{\infty}g_{k}^{r}\leq\left(\frac{2C}{n}\right)^{r}
\sum_{k=na_{n+1}+1}^{\infty}g_{k}^{r},\;\;n\in\mathbb{N}. 
\end{equation*}
Since $g_{k}>0$ for all $k\in\mathbb{N}$, the latter is impossible for
sufficiently large $n$. Therefore, \eqref{extra9} does not hold, and hence
we have \eqref{th_main_k_estimate}. Noting that \eqref{th_main_estimate_lp}
and \eqref{th_main_estimate_mon} are still satisfied, we complete the proof.

Finally, we consider the case when $p=0$ and $r=\infty$. The first part of
our process may be conducted in the same way as earlier. Namely, we can find
sequences of positive integers $\left(a_{n}\right)_{n=1}^{\infty}$, $
\left(b_{n}\right)_{n=1}^{\infty}$ and $\left(\gamma_{n}\right)_{n=1}^{
\infty}$ such that the sequence $\left(a_{n}\right)_{n=1}^{\infty}$ is
strictly increasing, $a_{n+1}\ge\gamma_{n}$ and the pointwise limit $
\widetilde{f}$ of the sequence of functions $\left(G_{n}\right)_{n=1}^{
\infty}$ defined by $G_{1}=T_{a_{1},b_{1}}(\widetilde{g})=\widetilde{g}$, $
G_{n}=T_{a_{n},b_{n}}(G_{n-1})$, $n\geq2$, satisfies conditions 
\eqref{th_main_estimate_lp} and \eqref{th_main_estimate_mon}. To get the
remaining condition \eqref{th_main_k_estimate}, we need somewhat to modify
(in fact, to simplify) the above procedure.

In this case the sequence $\left(\delta_{n}\right)_{n=1}^{\infty}$ is not needed, and we have only to arrange the choice of the sequence $
\left(b_{n}\right)_{n=1}^{\infty}$, which was arbitrary by now.

Since $g\in\ell^{q}\setminus\ell^{0}$, then $\mathrm{card}(\mathrm{supp}
\,g)=\infty$. Therefore, the function $\widetilde{g}(x)$ is strictly
positive for all $x>0$. Suppose that $n\geq1$ and $a_{k},b_{k},G_{k},
\gamma_{k}$ are defined for all $1\leq k\leq n$ (as above, $
a_{1}=b_{1}=\gamma_{1}=1$). To pass to the next step, we set 
\begin{equation}
a_{n+1}:=\max\left(\gamma_{n},a_{n}+1\right)\;\;\mbox{and}\;\;b_{n+1}:=\left[
\left(\frac{(n+1)G_n(a_{n+1})}{\widetilde{g}(n(a_{n+1}+1))}
\right)^{q}\right]+1.  \label{th_def_an_1-extra}
\end{equation}
Then, by the definition of the operators $T_{a,b}$, we have for all $n\ge2$ 
\begin{equation*}
G_{n}(a_{n}+1)\le b_{n}^{-1/q}G_{n-1}(a_{n})\le\frac{1}{n}\widetilde{g}
(n(a_{n}+1)). 
\end{equation*}
Combining this together with \eqref{extra2}, we get 
\begin{equation*}
\widetilde{f}(a_{n}+1)\le\frac{1}{n}\widetilde{g}(n(a_{n}+1)),\;\;n\ge2, 
\end{equation*}
which yields for all $n\ge2$ 
\begin{equation}
\lim\inf_{x\to\infty}\frac{\widetilde{f}(x/n)}{\widetilde{g}(x)}=0.
\label{extra10}
\end{equation}
Let us show that \eqref{extra10} implies \eqref{th_main_k_estimate} in the
case $p=0$ and $r=\infty$.

Assuming the contrary, for some positive $C$ we have
\begin{equation}
{\mathcal{K}}(t,g;\ell^{0},\ell^{\infty})\le C{\mathcal{K}}
(t,f;\ell^{0},\ell^{\infty}),\;\;t>0.  \label{extra11}
\end{equation}
This implies that 
\begin{equation*}
{\mathcal{K}}(t,\widetilde{g};L^{0}(0,\infty),L^{\infty}(0,\infty))\le C{
\mathcal{K}}(t,\widetilde{f};L^{0}(0,\infty),L^{\infty}(0,\infty)),\;\;t>0. 
\end{equation*}
Consequently, from \eqref{impl2} and \eqref{eq1b} it follows 
\begin{equation*}
\widetilde{g}(x)\le2C\widetilde{f}(x/2C)\;\;\mbox{for all}\;x>0. 
\end{equation*}
Since this inequality contradicts relation \eqref{extra10}, inequality 
\eqref{extra11} fails for any $C$. As a result,  \eqref{th_main_k_estimate} holds, and so the proof of Theorem \ref{Th-main-estimate} is completed.
\end{proof}

\begin{proof}[Proof of Lemma \protect\ref{Lemma_tab}]
Recall that $h$ is assumed to be a nonnegative, nonincreasing function in $
\Phi$ (i.e., constant on each interval of the form $\left[n-1,n\right)$, $
n\in\mathbb{N}$).

We begin with proving \eqref{lemma_tab_1}. Fix $a\in\mathbb{N}$ and $t>a.$
Then, for each $b\in\mathbb{N}$, we have 
\begin{eqnarray}
\int\nolimits _{0}^{t}\left(T_{a,b}h\left(x\right)\right)^{p}\,dx & = &
\int\nolimits
_{0}^{a}\left(T_{a,b}h\left(x\right)\right)^{p}\,dx+\int\nolimits
_{a}^{t}\left(T_{a,b}h\left(x\right)\right)^{p}\,dx  \notag \\
& = & \int\nolimits _{0}^{a}h\left(x\right)^{p}\,dx+\int\nolimits
_{a}^{t}\left(T_{a,b}h\left(x\right)\right)^{p}\,dx.  \label{extra 21}
\end{eqnarray}
Since 
\begin{eqnarray*}
0 & \leq & \int\nolimits
_{a}^{t}\left(T_{a.b}h\left(x\right)\right)^{p}\,dx=\int\nolimits
_{a}^{t}b^{-p/q}\left(h\left(a+\frac{x-a}{b}\right)\right)^{p}\,dx \\
& \leq & b^{-p/q}\left(t-a\right)h\left(0\right)^{p},
\end{eqnarray*}
it follows that 
\begin{equation*}
\lim_{b\rightarrow\infty}\int\nolimits
_{a}^{t}\left(T_{a.b}h\left(x\right)\right)^{p}\,dx=0. 
\end{equation*}
Combining this together with \eqref{extra 21}, we get \eqref{lemma_tab_1}.

To obtain \eqref{lemma_tab_2} and \eqref{lemma_tab_3}, note first that, for $
t\geq a$, the change of variables gives us 
\begin{equation}
\int\nolimits
_{t}^{\infty}\left(T_{a,b}h\left(x\right)\right)^{q}\,dx=\int\nolimits
_{a+\left(t-a\right)/b}^{\infty}h\left(x\right)^{q}\,dx.
\label{lemma_tab_est_1}
\end{equation}
Since $b\geq1$ we have that $t\geq a+\left(t-a\right)/b$ whenever $t\geq a$,
which together with \eqref{lemma_tab_est_1} implies \eqref{lemma_tab_2} in
the case $t\geq a.$ 
If $0<t<a$, then using the definition of the operator $T_{a,b}$ and 
\eqref{lemma_tab_est_1} for $t=a$, we obtain that 
\begin{eqnarray*}
\int\nolimits _{t}^{\infty}\left(T_{a,b}h\left(x\right)\right)^{q}\,dx & = &
\int\nolimits
_{t}^{a}\left(T_{a,b}h\left(x\right)\right)^{q}\,dx+\int\nolimits
_{a}^{\infty}\left(T_{a,b}h\left(x\right)\right)^{q}\,dx \\
& = & \int\nolimits _{t}^{a}h\left(x\right)^{q}\,dx+\int\nolimits
_{a}^{\infty}h\left(x\right)^{q}\,dx=\int\nolimits
_{t}^{\infty}h\left(x\right)^{q}\,dx,
\end{eqnarray*}
which establishes \eqref{lemma_tab_3} and also the remaining case of
inequality \eqref{lemma_tab_2}.

Next, for any fixed $a\in\mathbb{N}$, by the same change of variables as
above, we get 
\begin{equation*}
\int\nolimits
_{a}^{\infty}\left(T_{a,b}h\left(x\right)\right)^{r}\,dx=b^{1-r/q}\int
\nolimits _{a}^{\infty}h\left(x\right)^{r}\,dx. 
\end{equation*}
Hence, it follows that if $h\in L^{r}\left(0,\infty\right)$ then $
T_{a,b}h\in L^{r}\left(0,\infty\right)$ for all $a,b\in\mathbb{N}$.
Moreover, we see that \eqref{lemma_tab_4} holds if $h\in
L^{r}\left(0,\infty\right)$, since $0<q<r.$

Finally, we have to show that a positive integer $b$ satisfies 
\eqref{lemma_tab_5} whenever it satisfies \eqref{lemma_tab_5a}. First, for
each $t\in[0,a]$ it holds 
\begin{eqnarray*}
\int\nolimits _{t}^{\infty}\left(T_{a,b}h\left(x\right)\right)^{r}\,dx & = &
\int\nolimits
_{t}^{a}\left(T_{a,b}h\left(x\right)\right)^{r}\,dx+\int\nolimits
_{a}^{\infty}\left(T_{a,b}h\left(x\right)\right)^{r}\,dx \\
& = & \int\nolimits _{t}^{a}h\left(x\right)^{r}\,dx+b^{1-r/q}\int\nolimits
_{a}^{\infty}h\left(x\right)^{r}\,dx \\
& \leq & \int\nolimits _{t}^{\infty}h\left(x\right)^{r}\,dx,
\end{eqnarray*}
and so, for such values of $t$, \eqref{lemma_tab_5} is satisfied by each
positive integer $b$.

Further, we rewrite \eqref{lemma_tab_5a} in the form 
\begin{equation}
b^{1-r/q}\int\nolimits _{a}^{\infty}h\left(x\right)^{r}\,dx\leq\int\nolimits
_{t}^{\infty}h\left(x\right)^{r}\,dx.  \label{lemma_tab_est_2}
\end{equation}
Then, for $t>a$ we have 
\begin{eqnarray*}
\int\nolimits _{t}^{\infty}\left(T_{a,b}h\left(x\right)\right)^{r}\,dx &
\leq & \int\nolimits _{a}^{\infty}\left(T_{a,b}h\left(x\right)\right)^{r}\,dx
\\
& = & b^{1-r/q}\int\nolimits _{a}^{\infty}h\left(x\right)^{r}\,dx \\
& \leq & \int\nolimits _{t}^{\infty}h\left(x\right)^{r}\,dx.
\end{eqnarray*}
Therefore, \eqref{lemma_tab_5} is obtained for all $t>0$, and so the proof
of the lemma is completed.
\end{proof}

\vskip 1cm

\section{About the $S_{q}$-property}

In \cite{CN17}, in connection with the conjecture stated by Levitina,
Sukochev and Zanin (see Theorem \ref{int properties} and the subsequent
discussion), Cwikel and Nilsson have introduced the following notion.

\begin{definition}
Let $q\geq1$ and let $E\neq\{0\}$ be a normed sequence space, $
E\subseteq\ell^{q}.$ Then, $E$ has the \textit{$S_{q}$-property} provided
that there is a constant $C$ if, whenever $x=(x_{n})_{n=1}^{\infty}\in E$
and $y=(y_{n})_{n=1}^{\infty}\in\ell^{q}$ are two sequences, which satisfy
the conditions: 
\begin{equation}
\sum_{n=1}^{\infty}|x_{n}|^{q}=\sum_{n=1}^{\infty}|y_{n}|^{q}
\label{extra15}
\end{equation}
and 
\begin{equation}
\sum_{n=1}^{m}\left(x_{n}^{\ast}\right)^{q}\leq\sum_{n=1}^{m}\left(y_{n}^{
\ast}\right)^{q}\;\;\mbox{for all}\;m\in\mathbb{N},  \label{extra16}
\end{equation}
then it follows that $y\in E$ and $\left\Vert y\right\Vert _{E}\leq
C\left\Vert x\right\Vert _{E}$.
\end{definition}

It is clear that this definition may be extended to a more general situation when $q>0$ and $E$ is a quasi-Banach sequence space.

The following result shows that the $S_{q}$-property of a quasi-Banach
sequence lattice $E$ is closely related to the fact that $E\in
Int\left(\ell^{0},\ell^{q}\right)$.

\begin{theorem}
\label{S_q-prop} Let $0<q<\infty$ and $E$ be a quasi-Banach sequence
lattice. Then, the following conditions are equivalent:

(a) $E$ has the $S_{q}$-property;

(b) $E$ is a uniform ${\mathcal{K}}$-monotone space with respect to the
couple $\left(\ell^{0},\ell^{q}\right)$.

Therefore, from the condition (a) it follows that $E\in
Int\left(\ell^{0},\ell^{q}\right)$. In the case when $q\geq1$, the converse
holds as well, i.e., $E$ has the $S_{q}$-property if and only if $E\in
Int\left(\ell^{0},\ell^{q}\right)$.
\end{theorem}

\begin{proof}
$(a)\;\Longrightarrow\;(b)$. Assume that sequences $x=\left(x_{n}
\right)_{n=1}^{\infty}\in E$ and $y=\left(y_{n}\right)_{n=1}^{\infty}\in
\ell^{q}$ satisfy the condition 
\begin{equation*}
{\mathcal{K}}(t,y;\ell^{0},\ell^{q})\leq{\mathcal{K}}(t,x;\ell^{0},
\ell^{q}),\;t>0. 
\end{equation*}
Then, in the same way as in the end of Section \ref{Section-AC}, we have 
\begin{equation*}
\sum_{n=m}^{\infty}(y_{n}^{*})^{q}\le 2^q\sum_{n=m}^{
\infty}(D_{2}x^{*})_{n}^{q},\;\;m\in\mathbb{N}, 
\end{equation*}
or, denoting $u_{n}=2(D_{2}x^{*})_{n}$, $n=1,2,\dots$, 
\begin{equation}
\sum_{n=m}^{\infty}\left(y_{n}^{\ast}\right)^{q}\leq\sum_{n=m}^{
\infty}u_{n}^{q},\;\;\mbox{for all}\;m\in\mathbb{N}.  \label{extra22}
\end{equation}

Further, we will use a reasoning from the proof of Theorem 5.3 in \cite{CN17}. Since $E\subseteq\ell^{q}$, it follows that $\lim_{n\rightarrow
\infty}y_{n}=0.$ Select $n_{1}\in\mathbb{N}$ with $\left\vert
y_{n_{1}}\right\vert =y_{1}^{\ast}.$ Let $z=\left(z_{n}\right)_{n=1}^{\infty}
$ be a sequence such that $z_{n}=y_{n},n\neq n_{1},$ $z_1^*=|z_{n_1}|\ge y_1^*=|y_{n_1}|$ and $
\sum_{n=1}^{\infty}\left(z_{n}^{\ast}\right)^{q}=\sum_{n=1}^{
\infty}u_{n}^{q}.$ Then, by \eqref{extra22}, for all $m\in\mathbb{N}$ 
\begin{equation*}
\sum_{n=m}^{\infty}\left(z_{n}^{\ast}\right)^{q}\leq\sum_{n=m}^{
\infty}u_{n}^{q}, 
\end{equation*}
and hence we have 
\begin{eqnarray*}
\sum_{n=1}^{m}u_{n}^{q} & = &
\sum_{n=1}^{\infty}u_{n}^{q}-\sum_{n=m+1}^{\infty}u_{n}^{q} \\
& = &
\sum_{n=1}^{\infty}\left(z_{n}^{\ast}\right)^{q}-\sum_{n=m+1}^{
\infty}u_{n}^{q} \\
& \leq &
\sum_{n=1}^{\infty}\left(z_{n}^{\ast}\right)^{q}-\sum_{n=m+1}^{\infty}
\left(z_{n}^{\ast}\right)^{q}=\sum_{n=1}^{m}\left(z_{n}^{\ast}\right)^{q}.
\end{eqnarray*}
It is clear that every quasi-Banach sequence lattice, satisfying the $S_{q}$-property, is symmetric. Therefore, $u=(u_{n})\in E$ and $
\|u\|_{E}=2\|D_{2}x^{*}\|_{E}\le4\|x\|_{E}$. Consequently, since $E$ has the 
$S_{q}$-property, combining this together with the preceding relations, we
get that $\left(z_{n}\right)\in E$ and $\left\Vert z\right\Vert
_{E}\leq4C\left\Vert x\right\Vert _{E}$, where $C$ is the $S_{q}$-property
constant. Moreover,  by the definition of $z$, we have $|y|\leq|z|$.
Thus, since $E$ is a lattice, $\left\Vert y\right\Vert _{E}\leq\left\Vert z\right\Vert _{E}\leq4C\left\Vert
x\right\Vert _{E}$. Summarizing all, we conclude that $E$ is a uniform ${
\mathcal{K}}$-monotone space with respect to the couple $\left(\ell^{0},
\ell^{q}\right).$

$(b)\;\Longrightarrow\;(a)$. Let now $E$ be a uniform ${\mathcal{K}}$
-monotone space with respect to the couple $\left(\ell^{0},\ell^{q}\right).$
Suppose that sequences $x=(x_{n})_{n=1}^{\infty}\in E$ and $
y=(y_{n})_{n=1}^{\infty}\in\ell^{q}$ satisfy conditions \eqref{extra15} and 
\eqref{extra16}. Then, the same argument as in the first part of the proof
yields 
\begin{equation*}
\sum_{n=m}^{\infty}\left(y_{n}^{\ast}\right)^{q}\leq\sum_{n=m}^{\infty}
\left(x_{n}^{\ast}\right)^{q},\;\;m\in\mathbb{N}. 
\end{equation*}
This inequality, combined with formula \eqref{EQ13ddd} and implication 
\eqref{impl1}, yields 
\begin{equation*}
{\mathcal{K}}(t,y;\ell^{0},\ell^{q})\le{\mathcal{K}}(t,x;\ell^{0},\ell^{q}),
\;t>0. 
\end{equation*}
Hence, by the assumption, $y\in E$ and $\left\Vert y\right\Vert _{E}\leq
C\left\Vert x\right\Vert _{E}$, where $C$ is the ${\mathcal{K}}$
-monotonicity constant of $E$ with respect to the couple $
\left(\ell^{0},\ell^{q}\right)$.

Since every uniform ${\mathcal{K}}$-monotone space with respect to a couple
of quasi-Banach spaces is also an interpolation space with respect to this
couple, from (a) it follows that $E\in Int\left(\ell^{0},\ell^{q}\right)$.

Finally, assume that $q\geq1$. Then, by Corollary \ref{Ar-Cw}, every
interpolation space between $\ell^{0}$ and $\ell^{q}$ is a uniform ${
\mathcal{K}}$-monotone space. This fact, combined together with already
proved implication $(b)\Longrightarrow(a)$, implies the last assertion of
the theorem.
\end{proof}

From Theorems~\ref{S_q-prop} and \ref{int properties} we get

\begin{corollary}
\label{S_q-prop1} Let $q\geq1$. A quasi-Banach sequence lattice $E$ has
the $S_{q}$-property if and only if $E\in Int\left(\ell^{p},\ell^{q}\right)$
for some $p>0.$
\end{corollary}

Recall now the following definition from \cite{CN17}. For every sequence $
x=(x_{n})_{n=1}^{\infty}$ and each $N\in\mathbb{N}$, let $
(x_{n}^{(N)})_{n=1}^{\infty}$ be the truncated sequence defined by $
x_{n}^{(N)}=x_{n}$ if $1\leq n\leq N$ and $x_{n}^{(N)}=0$ if $n>N$. We say
that a normed sequence lattice $E$ has the \textit{weak Fatou property} if
there is a constant $R$ such that, for every sequence $x=(x_{n})_{n=1}^{
\infty}$ of nonnegative numbers with $(x_{n}^{(N)})_{n=1}^{\infty}\in E$ for
all $N\in\mathbb{N}$ and $\sup_{N\in\mathbb{N}}\Vert(x_{n}^{(N)})\Vert_{E}<
\infty$, we have $x\in E$ and 
\begin{equation*}
\left\Vert x\right\Vert _{E}\leq R\sup_{N\in\mathbb{N}}\left\Vert
(x_{n}^{(N)})\right\Vert _{E}. 
\end{equation*}
Obviously, each normed sequence lattice with the Fatou property (see Section \ref{Prel-sequence-spaces}) has the weak Fatou property.

According to the main result of \cite{CN17}, if $q>1$, then every normed
sequence lattice $E$ with the weak Fatou property has the $S_{q}$-property if
and only $E$ is an interpolation space between $\ell^{1}$ and $\ell^{q}$.
Moreover, by interpolation, from the assumption $E\in
Int\left(\ell^{1},\ell^{q}\right)$ it follows that $E\in
Int\left(\ell^{p},\ell^{q}\right)$ for all $0\leq p<1$. Therefore, applying
Theorem~\ref{S_q-prop}, we get the following result, which in a sense
complements Corollary \ref{triv cor}.

\begin{corollary}
Let $q>1$ and let $E$ be a Banach sequence lattice with the weak Fatou
property. Then, the following conditions are equivalent:

(i) $E\in Int\left(\ell^{p},\ell^{q}\right)$ for all $p\in[0,1)$;

(ii) $E\in Int\left(\ell^{p},\ell^{q}\right)$ for some $p\in[0,1)$;

(iii) $E\in Int\left(\ell^{1},\ell^{q}\right)$;

(iv) $E$ has the $S_{q}$-property.
\end{corollary}

\section{$\left(\ell^{p},\ell^{q}\right)$ is not a uniform 
Calder\'{o}n-Mityagin couple if $0\leq p<q<1$.}

It is a long-standing problem in the interpolation theory if a quasi-Banach couple with the Calder\'{o}n-Mityagin property possesses its uniform version as well (see, for instance, \cite[p. 1150]{KSM03}). In fact, by now this question is open when being restricted to the narrower classes of Banach couples or even of couples of Banach lattices.

In Section \ref{fails CM}, we proved that the couple $\left(\ell^{p},
\ell^{q}\right)$ does not have the Calder\'{o}n-Mityagin property whenever $
0\leq p<q<1$. As a consequence, we conclude that this couple fails to have
its uniform version. For the reader's convenience, we present here an
independent proof of the latter result, which is much shorter and simpler
than that of Theorem \ref{Th:No-CM-Property}.

\begin{theorem}
\label{Th:No-Bounded-CM-Property} The couple $\left(\ell^{p},\ell^{q}\right)$
, with $0\leq p<q<1$, does not have the uniform Calder\'{o}n-Mityagin property.
\end{theorem}

Let $\mathcal{P}_{p}$ and $\mathcal{Q}_{q}$ be the operators introduced in
Section \ref{Prel-sequence-spaces}. Taking into account the Holmstedt
formula \eqref{Holmsteds_formula4} if $p>0$ and relations \eqref{impl2} and 
\eqref{EQ13ddd} if $p=0$, one can easily see that Theorem \ref
{Th:No-Bounded-CM-Property} is a straightforward consequence of the
following proposition.

\begin{proposition}
\label{Prop_bounded_cm_property} Let $0\leq p<q<1.$ Then, given arbitrarily
large positive constant $C$ there exist two nonnegative, nonincreasing
sequences $x=\left(x_{n}\right)_{n=1}^{\infty}$ and $y=\left(y_{n}
\right)_{n=1}^{\infty}$ in $\ell^{q}$ satisfying the conditions 
\begin{equation}
(\mathcal{P}_{p}y)_{n}+n^{1/\alpha}(\mathcal{Q}_{q}y)_{n}\leq(\mathcal{P}
_{p}x)_{n}+n^{1/\alpha}(\mathcal{Q}_{q}x)_{n},\;\;n=1,2,\dots,\;\;\mbox{if}
\;p>0,  \label{equ3}
\end{equation}
where $1/\alpha=1/p-1/q$, and 
\begin{equation}
(\mathcal{Q}_{q}y)_{n}\leq(\mathcal{Q}_{q}x)_{n},\;\;n=1,2,\dots,\;\;
\mbox{if}\;p=0,  \label{equ4}
\end{equation}
such that for every linear operator $S:\,\ell^{q}\rightarrow\ell^{q}$ with $
Sx=y$ we have 
\begin{equation}
\left\Vert S\right\Vert _{\ell^{q}\to\ell^{q}}\geq C.
\label{Prop_bounded_cm_prop_norm_est}
\end{equation}
\end{proposition}

\begin{proof}
Taking for $y$ the element $e_{1}$ of the unit vector basis, we consider the
cases $p>0$ and $p=0$ separately.

(a) $p>0$. Since every ${\mathcal{K}}$-functional ${\mathcal{K}}
(t,x;X_{0},X_{1})$ is an increasing function in $t$, by 
\eqref{Holmsteds_formula4}, the sum $(\mathcal{P}_{p}z)_{n}+n^{1/\alpha}(
\mathcal{Q}_{q}z)_{n}$, for each $z\in\ell^{q}$, is almost increasing in $n$, i.e., 
\begin{equation}
(\mathcal{P}_{p}z)_{n}+n^{1/\alpha}(\mathcal{Q}_{q}z)_{n}\leq C_{p,q}\left((
\mathcal{P}_{p}z)_{m}+m^{1/\alpha}(\mathcal{Q}_{q}z)_{m}\right)\;\;\mbox{if}
\;n\le m,  \label{monotonicity}
\end{equation}
where $C_{p,q}\ge1$ depends only on $p$ and $q$.

Given any constant $C>0$, choose a positive integer $N$ so that 
\begin{equation}
(2C_{p,q})^{-1}N^{1/q-1}>C.  \label{choice}
\end{equation}
Next, we set $x=\left(x_{n}\right)_{n=1}^{\infty}$, where $
x_{n}=2C_{p,q}N^{-1/q}$ if $1\le n\le N$, and $x_{n}=0$ if $n>N$. Then, the first entry $(\mathcal{Q}_{q}x)_{1}$ of the sequence $\mathcal{Q}_{q}x$ is defined by
\begin{equation*}
(\mathcal{Q}_{q}x)_{1}=\Big(\sum_{n=1}^{N}x_{n}^{q}\Big)^{1/q}=2C_{p,q}. 
\end{equation*}
Therefore, by \eqref{monotonicity}, the right-hand side of inequality 
\eqref{equ3} is not less than $2$ for all $n\in\mathbb{N}$. On the other
hand, $(\mathcal{P}_{p}y)_{n}=1$, $n\in\mathbb{N}$, and $(\mathcal{Q}
_{q}y)_{1}=1$, $(\mathcal{Q}_{q}y)_{n}=0$, $n\ge2$. Hence, the left-hand side of \eqref{equ3} does not exceed $2$, and so for the above $x$ and $y$ inequality \eqref{equ3} holds.

Let now $S$ be a linear operator such that $S:\,\ell^{q}\rightarrow\ell^{q}$
with $Sx=y.$ Clearly, $S$ is defined by a sequence of bounded linear
functionals on $\ell^{q}$. In particular, setting $\Lambda\left(z\right):=
\left\langle Sz,e_{1}\right\rangle $, we have $\Lambda\left(x\right)=\left
\langle Sx,e_{1}\right\rangle =\left\langle y,e_{1}\right\rangle =1$ and 
\begin{equation*}
\left\vert \Lambda\left(z\right)\right\vert \leq\left\Vert S\right\Vert
\left\Vert z\right\Vert _{\ell^{q}}. 
\end{equation*}
Consequently, if $\beta_{n}:=\Lambda\left(e_{n}\right)$, $n\in\mathbb{N}$,
we have $\left\vert \beta_{n}\right\vert \leq\left\Vert S\right\Vert .$
Hence, 
\begin{equation*}
1=\left\vert \Lambda\left(x\right)\right\vert
=\Big\vert\sum_{j=1}^{N}\beta_{n}x_n\Big\vert\leq 2C_{p,q}\sum_{j=1}^{N}\beta_{n}N^{-1/q}\leq2C_{p,q}\left\Vert S\right\Vert
N^{1-1/q}. 
\end{equation*}
According to the choice of $N$ in \eqref{choice}, this implies that 
\begin{equation*}
\Vert S\Vert\geq(2C_{p,q})^{-1}N^{1/q-1}>C, 
\end{equation*}
and in this case the result follows.

(b) $p=0$. Given constant $C>0$, we take $N\in\mathbb{N}$ satisfying the inequality $N^{1/q-1}>C.$ Let $x=\left(x_{n}\right)_{n=1}^{\infty}$, where $
x_{n}=N^{-1/q}$ if $1\le n\le N$, and $x_{n}=0$ if $n>N$. As above, we have $
(\mathcal{Q}_{q}x)_{1}=1$. Therefore, since $(\mathcal{Q}_{q}y)_{1}=1$ and $(
\mathcal{Q}_{q}y)_{n}=0$ for all $n\ge2$, inequality \eqref{equ4} holds.

If $S$ is a linear operator such that $S:\,\ell^{q}\rightarrow\ell^{q}$ with 
$Sx=y$, the same reasoning as in the case (a) shows that $\|S\|\ge
N^{1/q-1}>C$, and so the proof is completed.
\end{proof}


\end{document}